\documentclass[11pt]{amsart}
\usepackage{amsmath,amssymb,amsthm,graphicx, url, verbatim, float}
\usepackage{enumerate}
\usepackage{color, relsize}
\usepackage[margin=1.20in, marginpar = 2.5cm]{geometry}

\theoremstyle{definition}
\newtheorem{defn}{Definition}[section]
\newtheorem{remark}[defn]{Remark}

\theoremstyle{plain}
\newtheorem{thm}[defn]{Theorem}

\newtheorem{prop}[defn]{Proposition}
\newtheorem{lem}[defn]{Lemma} 
\newtheorem{cor}[defn]{Corollary}

\newtheorem*{namedtheorem}{\theoremname}
\newcommand{\theoremname}{testing}
\newenvironment{named}[1]{\renewcommand{\theoremname}{#1}\begin{namedtheorem}}{\end{namedtheorem}}

\usepackage{mathtools}
\DeclarePairedDelimiter{\ceil}{\lceil}{\rceil}
\DeclarePairedDelimiter{\floor}{\lfloor}{\rfloor}

\newcommand{\jw}{\vcenter{\hbox{\includegraphics[scale=.1]{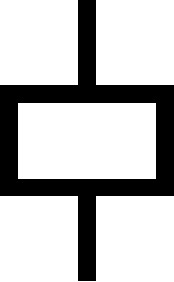}}}}
\newcommand{\cir}{  \vcenter{\hbox{\includegraphics[scale=.15]{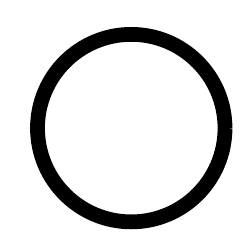}}}}
\newcommand{\cirj}{  \vcenter{\hbox{\includegraphics[scale=.15]{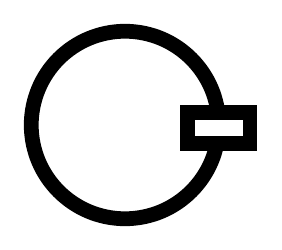}}}}
\newcommand{\sk}{\mathcal{S}}
\newcommand{\writhe}{\mathrm{wr}}

\newcommand{\sgn}{\mathrm{sgn}}
\newcommand{\Dj}{\mathbf{D}}
\newcommand{\dd}{\mathcal{D}}
\newcommand{\ac}{\mathcal{A}}

\begin{document}

\title{Jones diameter and crossing number of knots}

\author[E. Kalfagianni]{Efstratia Kalfagianni}
\author[C. Lee]{Christine Ruey Shan Lee}

\address[]{Department of Mathematics, Michigan State University, East
Lansing, MI, 48824}
\email[]{kalfagia@msu.edu}

\address[]{Department of Mathematics, Texas State University, San Marcos, TX, 78666}
\email[]{vne11@txstate.edu}

\begin{abstract}It has long been known that the quadratic term in the degree of the colored Jones polynomial of a knot is bounded above in terms of the crossing number of the knot.
We show that this bound is sharp if and only if the knot is adequate.

 As an application of our result we determine the crossing numbers of broad families of non-adequate prime satellite knots.  More specifically, we 
exhibit minimal crossing number  diagrams for  untwisted Whitehead doubles of zero-writhe adequate knots.  This allows us to determine
the crossing number of untwisted Whitehead doubles of any amphicheiral adequate knot, including, for instance, the Whitehead doubles 
of the connected sum of any alternating knot with its mirror image. 

We also determine the crossing number of the connected sum of  any adequate knot with an untwisted Whitehead double of a zero-writhe adequate knot.
 \end{abstract}

% AMS Subject classification: 57K10, 57K14, 57K16
% Keywords: Adequate knot, crossing number, colored Jones polynomial, satellite knots, Whitehead double

%%%%%%%%%%%%
\maketitle
\section{Introduction}

Given a knot $K$ let   $J_K(n)$
denote its $n$-th colored Jones polynomial, which is a Laurent polynomial in a variable $t$. Let  $d_+[J_K(n)]$ and $d_-[J_K(n)]$ denote the maximal and minimal degree of $J_K(n)$ in $t$ and set $$d[J_K(n)]:=4d_+[J_K(n)]-4d_-[J_K(n)].$$ The set of cluster points  $\left\{ n^{-2}d[J_K(n)]  \right\}'_{n{\in \mathbb N}}$ is known to be finite 
and the  point with the largest absolute value, denoted by $jd_K$, is called  the {\em Jones  diameter}  of $K$.  For precise definitions of the terms used here the reader is referred to Section 2.

Given a knot $K$ we will use $c(K)$ to denote the crossing number of $K$, the smallest number of crossings over all diagrams that represent $K$.
We prove the following.
\begin{thm} \label{main} Let $K$ be a knot with Jones diameter $jd_K$ and crossing number $c(K)$. Then,
$$jd_K \leq 2c(K),$$ 
with equality $jd_K =2c(K)$
if and only if $K$ is adequate.
\end{thm}

Adequate knots form a broad class that contains in particular all alternating knots.
 The works of Kauffman, Murasugi, and Thistlethwaite  \cite{Lickorishbook,  Kauffmanstates, Murasugi, alternating} imply that  for any knot $K$ we have  $jd_K \leq 2c(K),$ and that we have equality if $K$ is adequate. Our contribution here is
to show that if $jd_K = 2c(K)$, then $K$ must be adequate.

 Theorem \ref{main} has significant applications to the study of knot crossing  numbers. To state our result, recall
 that the
writhe of an adequate diagram $D=D(K)$ is an invariant of the knot $K$ \cite{Lickorishbook}.
We will use $\writhe(K)$ to denote this invariant.

\begin{thm}\label{doubleintro} For a knot $K$  with crossing number $c(K)$, let $W_{\pm}(K)$  denote its positive or negative untwisted Whitehead double.
Suppose that $K$ 
 is a non-trivial  adequate knot with $\mathrm{wr}(K)=0$.   Then,  $W_{\pm}(K)$ is non-adequate  and we have
$c(W(K))=4 c(K)+2$.
\end{thm}

Theorem \ref{doubleintro} should be compared with classical results in the literature asserting that the crossing numbers of several important classes of knots are realized by a ``special type" of knot diagrams. 
These classes include  alternating and more generally adequate knots, torus knots  and Montesinos knots \cite{ Kauffmanstates, Murasugi, alternating}.
The works of  Murasugi, Kauffman and Thistlethwaite, that settled the well known Tait  conjectures on alternating knots,  showed that adequate diagrams realize the crossing numbers of knots they represent. 
Moreover they showed that the crossing number is additive under connected sum of adequate knots. These results used in an essential way properties of the degree of the Jones polynomial.
The Jones polynomial was also used to derive lower bounds on crossing numbers of Whitehead doubles of adequate knots in an unpublished preprint of Pascual \cite{Pascual}. However, these bounds are not
strong enough to exactly determine the crossing numbers of any of these knots. For the proof of Theorem \ref{doubleintro} it is crucial that we have a sharper lower bound, that is also derived using the colored Jones polynomial,
and Theorem \ref{main}. See section 5 for details. To the best of our knowledge, Theorem \ref{doubleintro} and Corollary \ref{mirror} below are the first instances of results that  allow the exact determination of the crossing numbers for broad families of prime satellite knots.

Theorem  \ref{doubleintro}   applies  to adequate knots  that are equivalent to their mirror images (a.k.a. amphicheiral) since such knots must have $\writhe(K)=0$. One way to obtain an amphicheiral adequate knot is to take the connected sum of an adequate knot with its mirror image. For more discussion and examples of prime, amphicheiral, adequate knots see Section 5.

\begin{cor} \label{mirror}For a knot $K$ let $K^{*}$ denote the mirror image of $K$ and, for every  $m>0$, let $K_m:=\#^{m}(K\# K^{*})$ denote the connected sum of $m$-copies of $K\# K^{*}$. Suppose that $K$ is adequate with crossing number $c(K)$. Then, the untwisted Whitehead doubles $W_{\pm}(K_m)$ are non-adequate, and we have
$c(W_{\pm}(K))=8 m c(K)+2.$
\end{cor}

The method we use for the proof of Theorem  \ref{doubleintro}  also leads to the following application to the open conjecture on the additivity of crossing numbers  \cite[Problem 1.68]{Kirby}  under connected sums of knots. 
\begin{thm} \label{sum}
Suppose that $K$ is an adequate knot with $\mathrm{wr}(K)=0$ and let
  $K_1:=W_{\pm}(K)$. Then for any adequate knot $K_2$, the connected sum $K_1\# K_2$
is non-adequate and we have
  $$c(K_1\# K_2)=c(K_1)+c(K_2).$$
\end{thm}

Let us now briefly describe  the contents of the paper and our approach to proving the main theorems. It is known that  the degree of the colored Jones polynomial 
of a knot $K$ satisfies $d[J_K(n)]\leq 2c(D)n^2+O(n)$, for all $n\in {\mathbb N}$ and  any diagram $D=D(K)$ with $c(D)$ crossings.
Theorem \ref{main} will follow from a more general result,  Theorem \ref{t.mm},
stating that if the diagram  $D$ is not adequate then in fact we have $d[J_K(n)]\leq ( 2c(D)-q(D))n^2+O(n)$, where $q(D)$ is a positive constant depending on $D$.
Several terms used in the statement of  Theorem \ref{t.mm}  as well as in this  introduction are also defined in Section 2, where we also show how Theorem \ref{main} follows from
Theorem \ref{t.mm}.

 Sections 3 and 4 are devoted to the proof of Theorem \ref{t.mm}, which relies  on skein theoretic techniques and the fusion theory of the  $SU(2)$-quantum invariants for  knots and trivalent graphs. In Section 3 we include some background and preliminaries from these theories that we will use in the proof of Theorem \ref{t.mm}. For an outline of the proof and the  ideas involved, the reader is referred to the beginning of Section 4.2.

Theorem  \ref{doubleintro}, Corollary \ref{mirror}, and Theorem \ref{sum} are proved in Section 5.
Corollary \ref{criterion} of Theorem \ref{main} gives a criterion for determining the crossing number of a non-adequate knot provided that it admits a diagram whose number of crossings is close enough to the Jones  diameter of the  knot. 
The proof of Theorem \ref{doubleintro} uses this criterion and  a result
of Baker, Motegi and Takata \cite{BMT} that allows us to calculate the Jones diameter of Whitehead doubles.
We expect that Corollary \ref{criterion} will have similar applications for determining the crossing number of more classes of non-adequate knots.

We've made an effort to make the paper self-contained, by including some background and definitions, and by restating results we will use in the form we need them.
\vskip 0.06in

\noindent{ {\bf Acknowledgement.}} Kalfagianni acknowledges  partial research support NSF Grant, DMS-2004155 and Lee acknowledges  partial research support NSF Grant,
DMS-1907010. We thank Ken Baker and Kimihiko Motegi for their interest in this work and for helpful comments and questions on an earlier version of the paper.
We also thank the referee for their careful reading of an earlier version of our  manuscript and for their comments that helped to  improve our exposition and for noticing a slight oversight in the proof of Theorem 2.5.

\section{Degree bounds and Jones slope diameter}
\subsection{Diagrammatic degree bounds}Given a knot diagram $D$, a \emph{Kauffman state} \cite{KauffmanLins} is a choice of either the $A$-resolution or the $B$-resolution for each  crossing of $D$ as shown in Figure \ref{f.kstate}.  Applying a Kauffman state $\sigma$ to a diagram leads to a collection  $\sigma(D)$ of disjoint simple closed curve called \emph{state circles}.
The \emph{all-$A$} state on a knot diagram $D$, denoted by $\sigma_A$, is the state where the $A$-resolution is chosen at every crossing of $D$. Similarly, \emph{the all-$B$ state}, denoted by $\sigma_B$,  is the state where the $B$-resolution is chosen at every crossing of $D$.
\begin{figure}[H]
\def \svgwidth{.3\columnwidth}
%% Creator: Inkscape 1.0beta1 (32d4812, 2019-09-19), www.inkscape.org
%% Accompanies image file 'resolution.pdf' (pdf, eps, ps)
\begingroup%
  \makeatletter%
  \providecommand\color[2][]{%
    \errmessage{(Inkscape) Color is used for the text in Inkscape, but the package 'color.sty' is not loaded}%
    \renewcommand\color[2][]{}%
  }%
  \providecommand\transparent[1]{%
    \errmessage{(Inkscape) Transparency is used (non-zero) for the text in Inkscape, but the package 'transparent.sty' is not loaded}%
    \renewcommand\transparent[1]{}%
  }%
  \providecommand\rotatebox[2]{#2}%
  \newcommand*\fsize{\dimexpr\f@size pt\relax}%
  \newcommand*\lineheight[1]{\fontsize{\fsize}{#1\fsize}\selectfont}%
  \ifx\svgwidth\undefined%
    \setlength{\unitlength}{235.17197737bp}%
    \ifx\svgscale\undefined%
      \relax%
    \else%
      \setlength{\unitlength}{\unitlength * \real{\svgscale}}%
    \fi%
  \else%
    \setlength{\unitlength}{\svgwidth}%
  \fi%
  \global\let\svgwidth\undefined%
  \global\let\svgscale\undefined%
  \makeatother%
  \begin{picture}(1,0.71700104)%
    \lineheight{1}%
    \setlength\tabcolsep{0pt}%
    \put(0,0){\includegraphics[width=\unitlength,page=1]{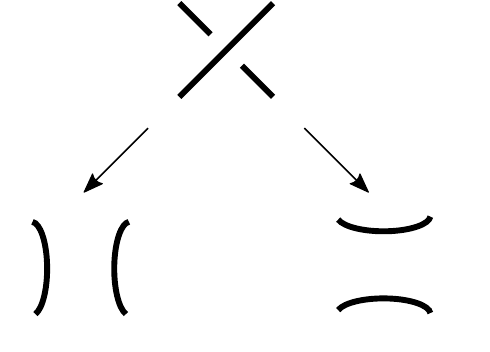}}%
    \put(-0.01024017,-0.01014027){\color[rgb]{0,0,0}\makebox(0,0)[lt]{\lineheight{1.25}\smash{\begin{tabular}[t]{l}$A$-resolution\end{tabular}}}}%
    \put(0.62121308,-0.01014027){\color[rgb]{0,0,0}\makebox(0,0)[lt]{\lineheight{1.25}\smash{\begin{tabular}[t]{l}$B$-resolution\end{tabular}}}}%
    \put(0,0){\includegraphics[width=\unitlength,page=2]{resolution.pdf}}%
  \end{picture}%
\endgroup%

\caption{ \label{f.kstate} The $A$- and $B$-resolution at a crossing. The dashed segments indicate the original location of the crossing. } 
\end{figure}

\begin{defn} \label{d.stategraph}
For a knot diagram $D$, the \textit{Kauffman state graph} $\mathbb{G}_\sigma(D)$ is the graph with vertices the set of state circles from applying $\sigma$ to $D$ and edges the dashed segments recording the original location of the crossing.  We will use  $v_{\sigma}(D)$ to denote the number of vertices of $\mathbb{G}_{\sigma}(D)$.
We write ${\mathbb G}_A$ to denote the state graph of  $\sigma_A$ and $v_A(D)$ its number of vertices.
Similarly, we will denote by ${\mathbb G}_B$ the state graph for $\sigma_B$ and its number of vertices by $v_B(D)$.

We define the following combinatorial quantities: 
\begin{itemize}
\item $c(D)$ is the number of crossings of the knot diagram $D$. 
\item Given an orientation on $D$, $c_+(D)$ and $c_-(D)$ are respectively the number of positive crossings and the number of negative crossings in the knot diagram $D$ with the conventions  specified in Figure \ref{f.pnc}.

\item $c_A(\sigma)$, respectively, $c_B(\sigma)$ is the number of crossings on which the Kauffman state $\sigma$ chooses the $A$-resolution, or respectively the $B$-resolution. 
\item \label{d.stategraphsign} $\sgn(\sigma)=c_A(\sigma) - c_B(\sigma)$ for a Kauffman state $\sigma$. 

\item The \emph{writhe} of a knot diagram $D$, denoted by $\writhe(D)$, is  $c_+(D) - c_-(D)$. 
\item The Turaev genus  $g_T(D)$ of $D$ is defined by $2g_T(D):=2-v_A(D)-v_B(D)+c(D)$ \cite{Turaevgenus, 5author}.
\end{itemize} 
\end{defn} 

\begin{figure}[H]
\def \svgwidth{.2\columnwidth}
%% Creator: Inkscape 1.0beta1 (32d4812, 2019-09-19), www.inkscape.org
%% Accompanies image file 'posnegcrossing.pdf' (pdf, eps, ps)
\begingroup%
  \makeatletter%
  \providecommand\color[2][]{%
    \errmessage{(Inkscape) Color is used for the text in Inkscape, but the package 'color.sty' is not loaded}%
    \renewcommand\color[2][]{}%
  }%
  \providecommand\transparent[1]{%
    \errmessage{(Inkscape) Transparency is used (non-zero) for the text in Inkscape, but the package 'transparent.sty' is not loaded}%
    \renewcommand\transparent[1]{}%
  }%
  \providecommand\rotatebox[2]{#2}%
  \newcommand*\fsize{\dimexpr\f@size pt\relax}%
  \newcommand*\lineheight[1]{\fontsize{\fsize}{#1\fsize}\selectfont}%
  \ifx\svgwidth\undefined%
    \setlength{\unitlength}{229.15473517bp}%
    \ifx\svgscale\undefined%
      \relax%
    \else%
      \setlength{\unitlength}{\unitlength * \real{\svgscale}}%
    \fi%
  \else%
    \setlength{\unitlength}{\svgwidth}%
  \fi%
  \global\let\svgwidth\undefined%
  \global\let\svgscale\undefined%
  \makeatother%
  \begin{picture}(1,0.42056259)%
    \lineheight{1}%
    \setlength\tabcolsep{0pt}%
    \put(0,0){\includegraphics[width=\unitlength,page=1]{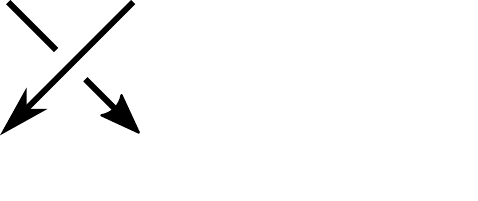}}%
    \put(0.04168163,0.01538434){\makebox(0,0)[lt]{\lineheight{1.25}\smash{\begin{tabular}[t]{l}$+1$\end{tabular}}}}%
    \put(0,0){\includegraphics[width=\unitlength,page=2]{posnegcrossing.pdf}}%
    \put(0.74862816,0.01538434){\makebox(0,0)[lt]{\lineheight{1.25}\smash{\begin{tabular}[t]{l}$-1$\end{tabular}}}}%
  \end{picture}%
\endgroup%

\caption{\label{f.pnc} A positive crossing and a negative crossing.}
\end{figure}  
\begin{defn}{\cite{LickorishThistlethwaite, Lickorishbook}} We say that a knot $K$ is  \emph{$A$-adequate}  if it admits a diagram $D=D(K)$ whose all-$A$ state graph has no one-edged loops. Such a diagram is called  \emph{$A$-adequate}
Similarly, a knot $K$ is \emph{$B$-adequate} if it admits a diagram whose all-$B$ state graph has no one-edged loops. A knot is \emph{adequate} if it admits a diagram $D=D(K)$   that is both $A$- and $B$-adequate. 
\end{defn}

Recall that  given a knot $K$ we use $J_K(n)$ to
denote its $n$-th colored Jones polynomial, which is a Laurent polynomial in a variable $t$.
Also $d_+[J_K(n)]$ and $d_-[J_K(n)]$ denotes the maximal and minimal degree of $J_K(n)$ in $t$, and
$d[J_K(n)]:=4d_+[J_K(n)]-4d_-[J_K(n)].$

 We will need the following well known lemma \cite{Lickorishbook, Indiana}.  

\begin{lem}\label{known} Given a knot diagram $D=D(K)$, for all $n\in \mathbb{N}$, we have the following.
\begin{enumerate}[(a)]
\item $d_+[J_K(n)] \leq \frac{c_+(D) }{2} n^2 + O(n)$  and we have equality if $D$ is $B$-adequate.
\item $d_-[J_K(n)]\geq -\frac{c_-(D)}{2} n^2 + O(n)$ and we have equality if $D$ is $A$-adequate.
\item $d[J_K(n)] \leq 2c(D)n^2 + (2-2g_T(D)  - 2c(D)) n + 2g_T(D) - 2$, and we have equality if $D$ is adequate.
\end{enumerate} 
\end{lem}

The Turaev genus of a knot $K$, denoted by $g_T(K)$, is defined to be the minimum $g_T(D)$  over all diagrams representing $K$.  It is known  \cite{Abe} that if $D$ is an adequate diagram
of a knot  $K$ then $g_T(K)=g_T(D)$.

 We have the following theorem which shows that   the diagrammatic bounds on  the degrees of the colored Jones polynomials given
 in Lemma \ref{known} are never achieved for non-adequate knot diagrams.

\begin{thm} \label{t.mm} Given $D=D(K)$  any diagram of a knot $K \subset S^3$ we have the following.
 \begin{enumerate}[(a)]
  \item \label{t.mma}  If $D$ is not $B$-adequate, then there is a constant $p_+(D)>0$ depending on $D$, such that
  $ d_+[J_K(n)]  \leq (\frac{c_+(D)}{2} -p_+(D)) n^2 + O(n) $, for all $n\in \mathbb{N}$.
 \item \label{t.mmb} If $D$ is not  $A$-adequate, then  there is a constant $p_-(D)>0$ depending on $D$, such that
 $ d_-[J_K(n)]  \geq (\frac{-c_-(D)}{2} +p_-(D)) n^2 + O(n)$,  for all $n\in \mathbb{N}$.
\item \label{t.mmc} If $D$ is not adequate, then there is a constant $p(D)>0$ depending on $D$, such that
 $ d[J_K(n)] \leq (2c(D)-4p(D))n^2 + O(n)  $, for all $n\in \mathbb{N}$. 
 \end{enumerate} 
\end{thm} 

The proof of Theorem \ref{t.mm} occupies the next two sections and, as we discuss next,  Theorem \ref{main} follows from it.

\subsection{Knots with maximal diameter}
Garoufalidis \cite{Garslopes} showed that for every knot $K$ there is a number $n_K > 0$ such that for $n>n_K$, the degrees $d_{\pm}[J_K(n)]$ are quadratic quasi-polynomials.
 That is, we have
\[4d_+[J_K(n)]  = s_2(n)n^2 + s_1(n)n + s_0(n) \ {\rm and }\ 4d_-[J_K(n)] = s^*_2(n)n^2 + s^*_1(n)n + s^*_0(n),  \] where for $0\leq i \leq 2$, $s_i, s^*_i:  \mathbb{N} \rightarrow \mathbb{Q}$ are periodic functions with integral period.

The elements of the sets 
$$js_K:= \left\{ s_2(n) \mid n>n_K\right\}\quad
 \mbox{and} \quad js^*_K:= \left\{ s^*_2(n)\mid  n>n_K \right\}$$
 are called {\em Jones slopes} of $K$. 
 Define the  {\em Jones slope diameter}  of $K$ by
 $$jd_K:= {\rm max}\{ |s_2(n)-s^*_2(n)| \mid n>n_K \ \}.$$ 

Let $c_+(K)$ (resp. $c_-(K)$) denote the minimum of $c_+(D)$ (resp. $c_-(D)$) over all knot diagrams of $K$.
Next we use Theorem \ref{t.mm} to derive the following result which in particular gives Theorem \ref{main} stated in the Introduction. 

\begin{thm}\label{slopes}Given a knot $K$ we have the following.
\begin{enumerate}[(a)]
\item $K$ is $B$-adequate  if and only if $2c_+(K)\in js_K$.
\item  $K$ is $A$-adequate  if and only if $-2c_-(K)\in js^{*}_K$.
\item $K$ is adequate if and only if  $jd_K = 2c(K).$
\end{enumerate}
\end{thm}
\begin{proof} One direction  of all three statements above is given in Lemma \ref{known}.
To deduce the other direction,
recall that for $n >  n_K$, $4d_+[J_K(n)], 4d_-[J_K(n)]$  become quadratic quasi-polynomials and let $s_2(n), s_2^{*}(n)$ denote their quadratic coefficients respectively \cite{Garslopes}. 
Recall that $s_2(n), s_2^{*}(n): \mathbb{N} \rightarrow \mathbb{Q}$ are periodic functions with integral period. 

If $2c_+(K)\in js_K$, then there is an infinite sequence $\{n_i>n_K\}_i\subset \mathbb{N}$ such that  $s_2(n_i)=2c_+(K)$.  Then Theorem \ref{t.mm} (a) implies that the diagram $D=D(K)$ that satisfies $c_+(K) = c_+(D)$ is 
$B$-adequate. Similarly, if $-2c_-(K)\in js^{*}_K$ then  by Theorem \ref{t.mm} (b) the diagram $D'=D'(K)$ that satisfies $c_-(K)=c_-(D')$ is 
$A$-adequate.

Suppose now that we have $jd_K = 2c(K).$ Then, there are Jones slopes $s_2\in js_K$ and $ s_2^*\in js^*_K$ so that  $|s_2-s_2^{*}|=2c(K)$. By the definitions,  there is an infinite sequence  $\{n_i>n_K\}_i$ so that
 $s_2(n_i)=s_2$  and $s_2^{*}(n_i)=s_2^{*}$ for every $n_i$.   Now let $D=D(K)$ be  a diagram that realizes $c(K)$ and let $c_-(D)$, $c_+(D)$ be the number of negative and positive crossings in $D$. We have $c(K)=c_-(D)+c_+(D)$ and
 $$4d_+[J_K(n_i)]-4d_-[J_K(n_i)]=(s_2-s_2^{*}) n_i^2+O(n_i)=2c(K)n_i^2+O(n_i),$$ for the infinite sequence $\{n_i>n_K\}_i$. Now Theorem \ref{t.mm} (c) implies that $D$ is adequate.
\end{proof}

\begin{remark} \label{r.throughe} The definition of Jones diameter used in \cite{Garslopes} is slightly different than the one
used in this paper. In  \cite{Garslopes}  the Jones diameter is defined to be the quantity
 $$ {\rm max}\{ |s_2-s^*_2| \mid s_2 \in js_K, s_2^* \in js^*_K\}.$$ 
Currently there are no examples of knots known for which the functions $s_2, s^*_2:  \mathbb{N} \rightarrow \mathbb{Q}$ ($n>n_K$) are not constant. Thus in all the cases where the Jones slopes have been computed the two definitions agree.
\end{remark}

\section{Fusion theory preliminaries and Tools} 
In this section, first, we recall some background from the skein and fusion theory of
the colored Jones polynomial, and restate some results from the literature that we will use in the proof of 
Theorem \ref{t.mm}.

\subsection{Kauffman brackets and skein theory}
\begin{defn} \emph{The Kauffman bracket skein module} $K(F)$ \cite{Przytycki}, \cite{KauffmanLins} of a compact, oriented surface $F$ with (possibly empty) boundary is the formal vector space over $\mathbb{C}(A)$ of properly embedded tangle or knot diagrams in $F$ (including the empty knot), considered up to isotopy fixing the boundary points, and modulo the \emph{Kauffman bracket skein relations}: 
\begin{itemize}
\item $  \vcenter{\hbox{\includegraphics[scale=.25]{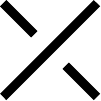}}}  = A^{-1} \ \vcenter{\hbox{\includegraphics[scale=.25]{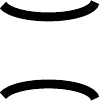}}}  + A \ \vcenter{\hbox{\includegraphics[scale=.25]{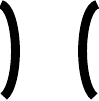}}} $
\item $  \vcenter{\hbox{\includegraphics[scale=.15]{circle.pdf}}} \sqcup D  = (-A^{-2}-{A^2})  D. $  
\end{itemize} 
An element of $K(F)$ is called a \emph{skein element}. Using the Kauffman bracket skein relations,  any skein element in $K(S^2)$ can be reduced to an empty diagram with coefficient a rational function in $\mathbb{C}(A)$. 
\begin{defn} Define the \emph{Kauffman bracket}   of the empty knot to be 1.
For a skein element  $D\in K(S^2)$,  the Kauffman bracket $\langle D \rangle$ is the rational function in $\mathbb{C}(A)$.
\end{defn} 

\end{defn} 
The \emph{Temperley-Lieb algebra} $TL_n$ is a specialization of the  Kauffman bracket skein module of the 2-disk $\dd^2$ viewed as a rectangle with $n$ marked points above and below. As a module, $TL_n$ is the vector space over $\mathbb{C}(A)$ of properly-embedded tangle or knot diagrams in $\dd^2$ such that the endpoints of a tangle are in 1-1 correspondence with the $2n$ marked points on $\partial(\dd^2)$, modulo the Kauffman bracket skein relations. 

The module $TL_n$ forms an algebra with the multiplication operation induced by stacking one disk on top of another, identifying the $n$ marked points on the bottom of the first disk with the $n$ marked points on top of the second disk, forming a new disk. The identity of the multiplication operation is the identity element $1_n$ of $n$ parallel arcs.  For two skein elements $\mathcal{U}, \mathcal{V}\in TL_n$, we will denote the result of the multiplication operation by $\mathcal{U}\cdot \mathcal{V}$. Every skein element in $TL_n$ can be written as a sum of products of elementary generators $e^n_1, \ldots, e^n_{n-1}$ by the Kauffman bracket skein relations.
\begin{figure}[H]
\def \svgwidth{.15\columnwidth}
%% Creator: Inkscape 1.0beta1 (32d4812, 2019-09-19), www.inkscape.org
%% Accompanies image file 'tlgen.pdf' (pdf, eps, ps)
\begingroup%
  \makeatletter%
  \providecommand\color[2][]{%
    \errmessage{(Inkscape) Color is used for the text in Inkscape, but the package 'color.sty' is not loaded}%
    \renewcommand\color[2][]{}%
  }%
  \providecommand\transparent[1]{%
    \errmessage{(Inkscape) Transparency is used (non-zero) for the text in Inkscape, but the package 'transparent.sty' is not loaded}%
    \renewcommand\transparent[1]{}%
  }%
  \providecommand\rotatebox[2]{#2}%
  \newcommand*\fsize{\dimexpr\f@size pt\relax}%
  \newcommand*\lineheight[1]{\fontsize{\fsize}{#1\fsize}\selectfont}%
  \ifx\svgwidth\undefined%
    \setlength{\unitlength}{229.33477687bp}%
    \ifx\svgscale\undefined%
      \relax%
    \else%
      \setlength{\unitlength}{\unitlength * \real{\svgscale}}%
    \fi%
  \else%
    \setlength{\unitlength}{\svgwidth}%
  \fi%
  \global\let\svgwidth\undefined%
  \global\let\svgscale\undefined%
  \makeatother%
  \begin{picture}(1,0.63935368)%
    \lineheight{1}%
    \setlength\tabcolsep{0pt}%
    \put(0,0){\includegraphics[width=\unitlength,page=1]{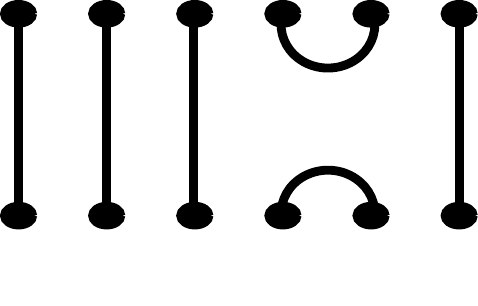}}%
    \put(0.51635129,0.00922335){\color[rgb]{0,0,0}\makebox(0,0)[lt]{\lineheight{1.25}\smash{\begin{tabular}[t]{l}$i$\end{tabular}}}}%
    \put(0.7256525,0.00922335){\color[rgb]{0,0,0}\makebox(0,0)[lt]{\lineheight{1.25}\smash{\begin{tabular}[t]{l}$i+1$\end{tabular}}}}%
  \end{picture}%
\endgroup%

\caption{An elementary generator $e^n_i$ where $n = 6$ and $i=4$. }
\end{figure} 
We will not explicitly mark the $n$ points on the boundary of the disk for $TL_n$ from this point on. 

\begin{defn} \label{d.jw} \cite{Wenzl}, \cite[Lemma 13.2]{Lickorishbook}  The \textit{Jones-Wenzl projector} in $TL_n$, denoted by $\jw_n$, is a unique element in $TL_n$ characterized by the following properties: 
\begin{enumerate}[(a)]
\item $\jw_n \cdot e^n_i= 0 = e^n_i \cdot \jw_n$ for $1\leq i \leq n-1$. 
\item $\jw_n - 1$ belongs to the algebra generated by $\{e^n_1, e^n_2, \ldots, e^n_{n-1} \}$. 
\item $\jw_n \cdot  \jw_n  = \jw_n $ 
\item \label{d.delta} Let $\cirj_n$ be the skein element in $K(S^2)$ obtained by embedding the disk $\dd^2$ containing $\jw_n$ in the standard way into $S^2$, and then joining the top $n$ points of $\dd^2$ to the bottom $n$ points by $n$ parallel arcs in the projection plane. The Kauffman bracket of $\cirj_n$ is given by 
$\langle \cirj_n  \rangle = (-1)^n \frac{A^{2(n+1)} - A^{-2(n+1)}}{A^2-A^{-2}}.$

To simplify notation we will denote  $\triangle_n:=\langle \cirj_n  \rangle$.
\end{enumerate} 
We will depict a skein element in $TL_n$ containing Jones-Wenzl projectors by drawing rectangular boxes and say that the skein element is \textit{decorated }by a Jones-Wenzl projector if it contains a Jones-Wenzl projector. 
\end{defn} 

\subsection{Fusion and untwisting formulas}
 For a skein element in $K(S^2)$, we shall take for granted the fusion and untwisting formulas from \cite{MasbaumVogel}, depicted in Figure \ref{f.fusion}. As in \cite{MasbaumVogel}, a trivalent graph colored with nonnegative integers represents a skein element containing Jones-Wenzl projectors by making the local replacement shown below in Figure \ref{f.trigraph} at a vertex of the trivalent graph. 
 \begin{figure}[H]
 \begin{center} 
 \def \svgwidth{.24\columnwidth}
 %% Creator: Inkscape 1.0beta1 (32d4812, 2019-09-19), www.inkscape.org
%% Accompanies image file 'trigraph.pdf' (pdf, eps, ps)
\begingroup%
  \makeatletter%
  \providecommand\color[2][]{%
    \errmessage{(Inkscape) Color is used for the text in Inkscape, but the package 'color.sty' is not loaded}%
    \renewcommand\color[2][]{}%
  }%
  \providecommand\transparent[1]{%
    \errmessage{(Inkscape) Transparency is used (non-zero) for the text in Inkscape, but the package 'transparent.sty' is not loaded}%
    \renewcommand\transparent[1]{}%
  }%
  \providecommand\rotatebox[2]{#2}%
  \newcommand*\fsize{\dimexpr\f@size pt\relax}%
  \newcommand*\lineheight[1]{\fontsize{\fsize}{#1\fsize}\selectfont}%
  \ifx\svgwidth\undefined%
    \setlength{\unitlength}{345.3911897bp}%
    \ifx\svgscale\undefined%
      \relax%
    \else%
      \setlength{\unitlength}{\unitlength * \real{\svgscale}}%
    \fi%
  \else%
    \setlength{\unitlength}{\svgwidth}%
  \fi%
  \global\let\svgwidth\undefined%
  \global\let\svgscale\undefined%
  \makeatother%
  \begin{picture}(1,0.486562)%
    \lineheight{1}%
    \setlength\tabcolsep{0pt}%
    \put(0,0){\includegraphics[width=\unitlength,page=1]{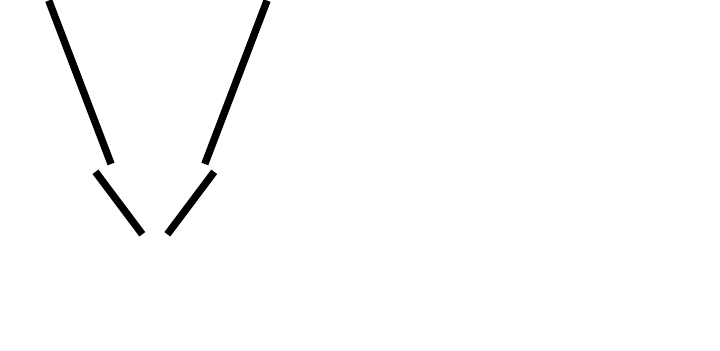}}%
    \put(-0.00631564,0.44349895){\color[rgb]{0,0,0}\makebox(0,0)[lt]{\lineheight{0}\smash{\begin{tabular}[t]{l}$a$\end{tabular}}}}%
    \put(0.3834802,0.44349895){\color[rgb]{0,0,0}\makebox(0,0)[lt]{\lineheight{0}\smash{\begin{tabular}[t]{l}$b$\end{tabular}}}}%
    \put(0,0){\includegraphics[width=\unitlength,page=2]{trigraph.pdf}}%
    \put(0.2545659,0.0431283){\color[rgb]{0,0,0}\makebox(0,0)[lt]{\lineheight{0}\smash{\begin{tabular}[t]{l}$c$\end{tabular}}}}%
    \put(0,0){\includegraphics[width=\unitlength,page=3]{trigraph.pdf}}%
    \put(0.47611425,0.1992883){\makebox(0,0)[lt]{\lineheight{1.25}\smash{\begin{tabular}[t]{l}$=$\end{tabular}}}}%
    \put(0.95239979,0.44349895){\color[rgb]{0,0,0}\makebox(0,0)[lt]{\lineheight{0}\smash{\begin{tabular}[t]{l}$b$\end{tabular}}}}%
    \put(0.82348575,0.05615704){\color[rgb]{0,0,0}\makebox(0,0)[lt]{\lineheight{0}\smash{\begin{tabular}[t]{l}$c$\end{tabular}}}}%
    \put(0.60169015,0.44349895){\color[rgb]{0,0,0}\makebox(0,0)[lt]{\lineheight{0}\smash{\begin{tabular}[t]{l}$a$\end{tabular}}}}%
  \end{picture}%
\endgroup%

 \end{center} 
 \caption{\label{f.trigraph} A skein element (left) is recovered from a colored trivalent graph (right) by replacing a vertex locally by the skein element in $K(\dd^2)$ with $a+b$ marked points above and $c$ points below.}
 \end{figure} 
 
\begin{figure}[H]
\def \svgwidth{0.95\columnwidth}
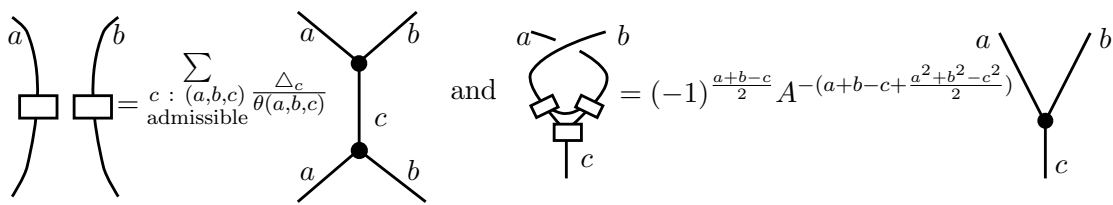 
\caption{\label{f.fusion} The fusion (left) and untwisting (right) formulas.}
\end{figure} 

We call a triple of nonnegative integers $(a, b, c)$ \emph{admissible} if $a+b+c$ is even, $a\leq b+c$, $b\leq a+c$, and $c\leq a+b$.
Given an admissible triple, the rational function $\theta(a, b, c)$ is defined as the Kauffman bracket of a particular skein element as shown in Figure \ref{Theta} below:
\begin{figure}[H]
\def\svgwidth{.25\columnwidth}
\[\theta(a, b, c) = \left \langle\ \vcenter{\hbox{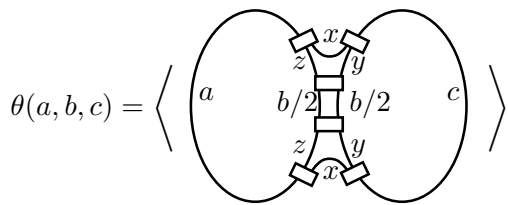}} \  \right \rangle \]
\caption{We have $x = \frac{a+c-b}{2}, y=\frac{b+c-a}{2}, z=\frac{a+b-c}{2}$.}\label{Theta}
\end{figure} 

\begin{lem}[{\cite[Lemma 14.5]{Lickorishbook}}] Suppose the triple of nonnegative integers  $(a, b, c)$ is admissible. Then, we have 
\begin{equation}\label{e.theta} 
\theta(a, b, c) = \frac{\triangle_{x+y+z}! \triangle_{x-1}! \triangle_{y-1}! \triangle_{z-1}!}{\triangle_{y+z-1}! \triangle_{z+x-1}! \triangle_{x+y-1}!}, 
 \end{equation} 
 where $\triangle_n!:=\triangle_n \triangle_{n-1} \triangle_{n-2} \cdots \triangle_{1}$ and $\triangle_{-1} = \triangle_{0} := 1$.
 \end{lem}

\subsection{A definition of the colored Jones polynomial}
Given a knot diagram $D \subset S^2$ with $\ell$ components, let $\ac = \cup_{i=1}^\ell \ac_i$ be the collection of annuli in $S^2$ containing $D$ with boundaries the 2-blackboard cable of the diagram $D$. The natural inclusion map $\iota: \ac \hookrightarrow S^2$ induces a map $\iota^*: K(\ac) \hookrightarrow K(S^2)$. 

Define the skein element $\Dj^n$ as the image of the element $(\cirj_n,\cirj_n, \ldots, \cirj_n )$ in $K(\ac_1) \times K(\ac_2) \times \cdots \times K(\ac_\ell)$  under the map $\iota^*$. That is,
\[\Dj^n: = \iota^*(\cirj_n,\cirj_n, \ldots, \cirj_n ).  \]

For depiction of skein elements  we follow the convention of \cite{MasbaumVogel, Lickorishbook} where a label $n$
next to a strand indicates number of parallel strands. We  use boldface for $\Dj^n$ to denote 
$n$-blackboard cable $D^n$ of $D$, decorated by Jones-Wenzl projectors as defined above.
\begin{figure}[H] 
\def\svgwidth{.6\columnwidth}
%% Creator: Inkscape 1.0beta1 (32d4812, 2019-09-19), www.inkscape.org
%% Accompanies image file 'cable.pdf' (pdf, eps, ps)
\begingroup%
  \makeatletter%
  \providecommand\color[2][]{%
    \errmessage{(Inkscape) Color is used for the text in Inkscape, but the package 'color.sty' is not loaded}%
    \renewcommand\color[2][]{}%
  }%
  \providecommand\transparent[1]{%
    \errmessage{(Inkscape) Transparency is used (non-zero) for the text in Inkscape, but the package 'transparent.sty' is not loaded}%
    \renewcommand\transparent[1]{}%
  }%
  \providecommand\rotatebox[2]{#2}%
  \newcommand*\fsize{\dimexpr\f@size pt\relax}%
  \newcommand*\lineheight[1]{\fontsize{\fsize}{#1\fsize}\selectfont}%
  \ifx\svgwidth\undefined%
    \setlength{\unitlength}{520.36725863bp}%
    \ifx\svgscale\undefined%
      \relax%
    \else%
      \setlength{\unitlength}{\unitlength * \real{\svgscale}}%
    \fi%
  \else%
    \setlength{\unitlength}{\svgwidth}%
  \fi%
  \global\let\svgwidth\undefined%
  \global\let\svgscale\undefined%
  \makeatother%
  \begin{picture}(1,0.32269544)%
    \lineheight{1}%
    \setlength\tabcolsep{0pt}%
    \put(0,0){\includegraphics[width=\unitlength,page=1]{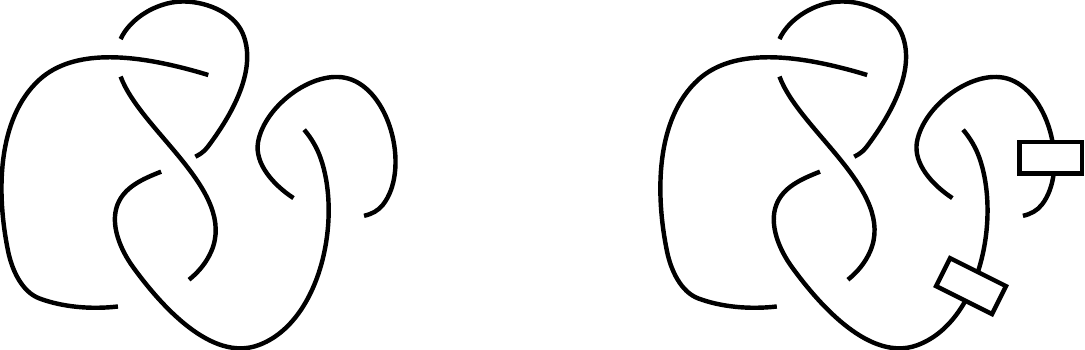}}%
    \put(0.5696536,0.19153911){\makebox(0,0)[lt]{\lineheight{1.25}\smash{\begin{tabular}[t]{l}$n$\end{tabular}}}}%
    \put(0.91556314,0.2636036){\makebox(0,0)[lt]{\lineheight{1.25}\smash{\begin{tabular}[t]{l}$n$\end{tabular}}}}%
  \end{picture}%
\endgroup%

\caption{Left: An example of a knot diagram $D$ and the resulting $\Dj^n$. }
\end{figure} 

\begin{defn} \label{d.cjp} The (unreduced) $n$-th colored Jones polynomial of the knot $K$, with a diagram $D=D(K)$ is given by
$J_K(n)(t) =  ((-1)^{n-1} A^{n^2-1})^{\writhe(D)} \langle \Dj^{n-1} \rangle \rvert_{A = t^{-1/4}},$ where
for the unknot $U = \cir$ we have  $J_U(n)(t) = (-1)^{n-1} \frac{t^{-n/2}-t^{n/2}}{t^{-1/2}-t^{1/2}}$ for $n\geq 2$.

To simplify the notation we will omit the variable $t$ and write $J_K(n)$ for $J_K(n)(t)$.
\end{defn}

\subsection{Fusion calculus and degree bounds on Kauffman brackets of skein elements}
Let $f(A)$ be a rational function $\frac{P(A)}{Q(A)}$ with $P(A), Q(A)$ polynomials with complex coefficients in the variable $A$. We define $\deg f(A)$ to be the maximum power of $A$ in the formal Laurent series expansion of $f(A)$ whose $A$-power is bounded from above. 

For example, to compute $\deg(\frac{A}{-A^2 - A^{-2}})$ we write 
$\frac{A}{-A^2 - A^{-2}} = \frac{A}{-A^{2}(1+A^{-4})} = -A^{-1}\frac{1}{1+A^{-4}}$.
Then, we expand $1/(1+A^{-4}) =  1-A^{-4} + A^{-8} -\cdots$ and multiply by $-A^{-1}$ to obtain 
$ \frac{A}{-A^2 - A^{-2}}   = -A^{-1} + A^{-5} - A^{-9} + \cdots,$
so $\deg(\frac{A}{-A^2 - A^{-2}}) = -1$.    

Note $\deg f(A) = \deg P(A) - \deg Q(A).$ 
With this example of $f(A)$, if we factor the denominator $Q(A)$ in the other way: $Q(A) = -A^{-2}(A^4 + 1)$, then we would obtain the Laurent series expansion of $f(A)$ whose minimum $A$-power is bounded from below which can be used to find the minimum degree of $f(A)$. 

We are interested in using the degrees of rational function summands from the Kauffman bracket definition of the $n$-th colored Jones polynomial to estimate the degree.  To that end, it will be useful to keep in mind the degrees of the following rational functions. 

\begin{lem}\label{l.basic}  We have the following.
\begin{enumerate}[(a)]
\item  $\deg \triangle_c = 2c$.
\item $\deg \theta(a, b, c) = a + b + c$.
\item \label{l.basicc}$ \deg \frac{\triangle_c}{\theta(a, b, c)} = 2c - (a+b+c) = c-a-b$.
\end{enumerate}
\end{lem}
\begin{proof} 
To compute the degree of a rational function $P(A)/Q(A)$, we take the difference $\deg P(A) - \deg Q(A)$. This is shown in statement \eqref{l.basicc} of the lemma for the degree of $\frac{\triangle_c}{\theta(a, b, c)}$ after determining $\deg \triangle_c $ and $\deg \theta(a, b, c)$. So we compute 
$\deg \triangle_c$. By Definition \ref{d.jw} \eqref{d.delta}, 
\[\triangle_c = (-1)^c \frac{A^{2(c+1)} - A^{-2(c+1)}}{A^2-A^{-2}}.\] 
Therefore $\deg \triangle_c= 2(c+1) - 2 = 2c. $
Next from the explicit formula of $\theta(a, b, c)$ from Equation \eqref{e.theta}, we get 
\[ \deg \theta = \deg (\triangle_{x+y+z}! \triangle_{x-1}! \triangle_{y-1}! \triangle_{z-1}!)  - \deg (\triangle_{y+z-1}! \triangle_{z+x-1}! \triangle_{x+y-1}!),  \] 
with $x = \frac{a+c-b}{2}, y=\frac{b+c-a}{2}, z=\frac{a+b-c}{2}$. 
Note $\deg (\triangle_c!) = c(c+1)$ by summing over the degrees in the factorial. Plugging this into the equation above and simplifying gives the desired result.
\end{proof}

A \emph{twist region} $T$ is a $(2, 2)$-tangle of crossing(s) arranged end-to-end. If there is more than one crossing, then we require that the crossings alternate. For example, $\vcenter{\hbox{\includegraphics[scale=.25]{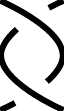}}}$ and $\vcenter{\hbox{\includegraphics[scale=.25]{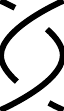}}}$ are both twist regions with two crossings, and $\vcenter{\hbox{\includegraphics[scale=.25]{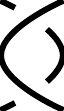}}}$ is not.

\begin{figure}[H]
 \def \svgwidth{.075\columnwidth}
 %% Creator: Inkscape 1.0beta1 (32d4812, 2019-09-19), www.inkscape.org
%% Accompanies image file 'J.pdf' (pdf, eps, ps)
\begingroup%
  \makeatletter%
  \providecommand\color[2][]{%
    \errmessage{(Inkscape) Color is used for the text in Inkscape, but the package 'color.sty' is not loaded}%
    \renewcommand\color[2][]{}%
  }%
  \providecommand\transparent[1]{%
    \errmessage{(Inkscape) Transparency is used (non-zero) for the text in Inkscape, but the package 'transparent.sty' is not loaded}%
    \renewcommand\transparent[1]{}%
  }%
  \providecommand\rotatebox[2]{#2}%
  \newcommand*\fsize{\dimexpr\f@size pt\relax}%
  \newcommand*\lineheight[1]{\fontsize{\fsize}{#1\fsize}\selectfont}%
  \ifx\svgwidth\undefined%
    \setlength{\unitlength}{318.12316654bp}%
    \ifx\svgscale\undefined%
      \relax%
    \else%
      \setlength{\unitlength}{\unitlength * \real{\svgscale}}%
    \fi%
  \else%
    \setlength{\unitlength}{\svgwidth}%
  \fi%
  \global\let\svgwidth\undefined%
  \global\let\svgscale\undefined%
  \makeatother%
  \begin{picture}(1,0.75442476)%
    \lineheight{1}%
    \setlength\tabcolsep{0pt}%
    \put(0,0){\includegraphics[width=\unitlength,page=1]{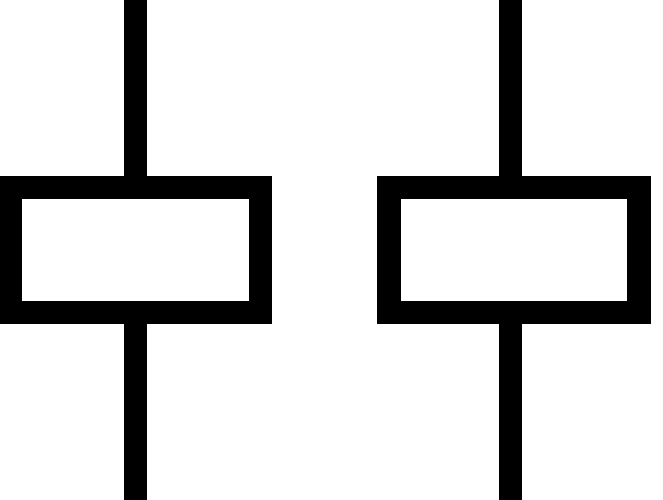}}%
    \put(-0.17260062,0.66012169){\color[rgb]{0,0,0}\makebox(0,0)[lt]{\lineheight{1.25}\smash{\begin{tabular}[t]{l}$n$\end{tabular}}}}%
    \put(0.83172918,0.66012169){\color[rgb]{0,0,0}\makebox(0,0)[lt]{\lineheight{1.25}\smash{\begin{tabular}[t]{l}$n$\end{tabular}}}}%
  \end{picture}%
\endgroup%

 \caption{\label{f.J} The skein element $\mathcal{J}_n$.}
\end{figure} 

We will consider 
$T$ as a skein element in $TL_2$. 
Let $\mathcal{J}_n$ be the skein element in $TL_{2n}$ that is two Jones-Wenzl projectors $\jw_n$ placed side by side, see Figure \ref{f.J}. Let $T^n \in TL_{2n}$ be the $n$-blackboard cable of $T$ and let $\mathbf{T}^n = \mathcal{J}_n \cdot T^n \cdot \mathcal{J}_n$.   

\begin{figure}[H]
\begin{center} 
\def \svgwidth{\columnwidth}
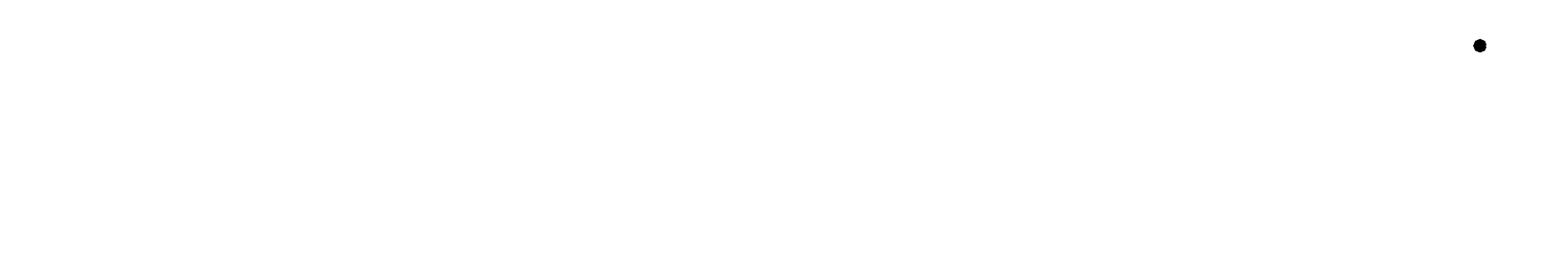 
\end{center} 
\caption{\label{f.funtwist} Fusion on a twist region.}
\end{figure} 

\begin{lem}\label{l.Twistdegree} Given an admissible triple $(a, n, n)$ and an integer $r\neq 0$ , let $d(a, r, n)$ denote the degree of the coefficient function $I(a, r, n): =\frac{\triangle_n}{\theta(n, n, a)} (-1)^{n-\frac{a}{2}} A^{r(2n-a + n^2 -\frac{a^2}{2})} $ resulting from first applying the fusion and then the untwisting formula to evaluate $\langle \mathbf{T}^n \rangle$  from a twist region $T$ with $r$ crossings. Then, we have
$$ d(a, r, n): = \deg I(a, r, n) = 2(r-1)n + (1-r)a + rn^2 - r\frac{a^2}{2}.$$
The case $r=4$ is illustrated in Figure \ref{f.funtwist}.
\end{lem}

\begin{proof}
This is a straightforward application of the fusion and untwisting formulas and Lemma \ref{l.basic} for the degree of $\frac{\triangle_n}{\theta(n, n, a)}$.  
\end{proof} 

\begin{defn}\label{notation}
We will denote the skein element in $TL_{2n}$ in the sum resulting from applying the fusion and untwisting formulas to $\mathbf{T}^n$ by $\mathcal{I}(a, r, n)$.
\end{defn}

Let $\sk$ be a skein element in $TL_n$ with crossings which may or may not contain a Jones-Wenzl projector. 
The definition of Kauffman states on knot diagrams with crossings extends naturally to skein elements in $TL_n$ with crossings.
We denote by $\sigma(\sk)$ the skein element resulting from applying a Kauffman state $\sigma$ to the crossings of $\sk$, and $\mathbb{G}_\sigma(\sk)$ will denote $\sigma(\sk)$ with the dashed segments of Figure \ref{f.kstate}, which will still be called ``edges" even though $\mathbb{G}_\sigma(\sk)$ may no longer be a graph. For a skein element $\sk$ that is a knot diagram in $K(S^2)$, $\sigma(\sk)$ is the same as the set of state circles. The collection of arcs and circles in $\sigma(\sk)$  are called the \textit{state arcs and state circles} of $\sigma(\sk)$. 

\subsection{Some useful lemmas}Here we will restate some technical lemmas  and definitions from  \cite{lee2020jones}  and we will prove an auxiliary lemma that we will be using in the proof of Theorem \ref{t.mm}. The reader may choose to move directly to the proof of the theorem in the next section and return to the statements as they are called in the course of the proof of Theorem \ref{t.mm}.

\begin{defn} Given a positive integer $k$, let $P = \{k_1, \ldots, k_s\}$ be a nonnegative integer partition of $k$ into $s$ parts. We say that $P$ is a \emph{minimal partition}, denoted by $P_m$, if it minimizes the quantity $m(P)  = \max_{1\leq i \leq s}\{k_i\}$ over all partitions $P$ of $k$ into $s$ parts. 

For example, $P_m = \{1, 2, 2\}$ is a minimal partition of $5=1+2+2$ into 3 parts, and $P = \{3, 2, 0\}$ is not a minimal partition of 5 into 3 parts.   
\end{defn}
A minimal partition exists by the following elementary lemma.
\begin{lem}\label{elementarym}
Given $k$ and $s$, one can find a minimal partition by writing $k = \mu s + b$, where $\mu, b$ are positive integers and $b<s$ using Euclidean division.  A minimal partition of $k$ is given by:
\begin{equation} \label{e.minpartition}
\{\underbrace{\ceil{k/s} = k/s + (s-b)/s, \ldots, \ceil{k/s}}_{\text{$b$ times}},  \underbrace{\floor{k/s} = k/s - b/s, \ldots, \floor{k/s}}_{\text{$s-b$ times}}\}. 
\end{equation} 
\end{lem}

Next we recall the following lemma.
\begin{lem}[{\cite[Lemma 3.12]{lee2020jones}}] \label{l.partition} Fix $k$ and $s$. The minimal partition $P_m = \{m_1, \ldots, m_s \}$ of $k$ into $s$ parts is unique up to rearrangement of indices. If $P = \{k_1, \ldots, k_s\}$ is another partition of $k$ into $s$ parts, then 
\[\sum_{i=1}^s m_i^2 \leq \sum_{i=1}^s k_i^2. \] 
\end{lem} 

The next lemma provides conditions under which
a skein element that arises in the evaluation of the Kauffman bracket of $\langle \Dj^n \rangle$, by applying the fusion and the untwisting formula to the $n$-cable of a twist region in $\Dj^n$, evaluates to zero. 
\begin{figure}[H]
\def \svgwidth{.15\columnwidth}
%% Creator: Inkscape 1.0beta1 (32d4812, 2019-09-19), www.inkscape.org
%% Accompanies image file 'dskein.pdf' (pdf, eps, ps)
\begingroup%
  \makeatletter%
  \providecommand\color[2][]{%
    \errmessage{(Inkscape) Color is used for the text in Inkscape, but the package 'color.sty' is not loaded}%
    \renewcommand\color[2][]{}%
  }%
  \providecommand\transparent[1]{%
    \errmessage{(Inkscape) Transparency is used (non-zero) for the text in Inkscape, but the package 'transparent.sty' is not loaded}%
    \renewcommand\transparent[1]{}%
  }%
  \providecommand\rotatebox[2]{#2}%
  \newcommand*\fsize{\dimexpr\f@size pt\relax}%
  \newcommand*\lineheight[1]{\fontsize{\fsize}{#1\fsize}\selectfont}%
  \ifx\svgwidth\undefined%
    \setlength{\unitlength}{257.01868012bp}%
    \ifx\svgscale\undefined%
      \relax%
    \else%
      \setlength{\unitlength}{\unitlength * \real{\svgscale}}%
    \fi%
  \else%
    \setlength{\unitlength}{\svgwidth}%
  \fi%
  \global\let\svgwidth\undefined%
  \global\let\svgscale\undefined%
  \makeatother%
  \begin{picture}(1,0.89204089)%
    \lineheight{1}%
    \setlength\tabcolsep{0pt}%
    \put(0,0){\includegraphics[width=\unitlength,page=1]{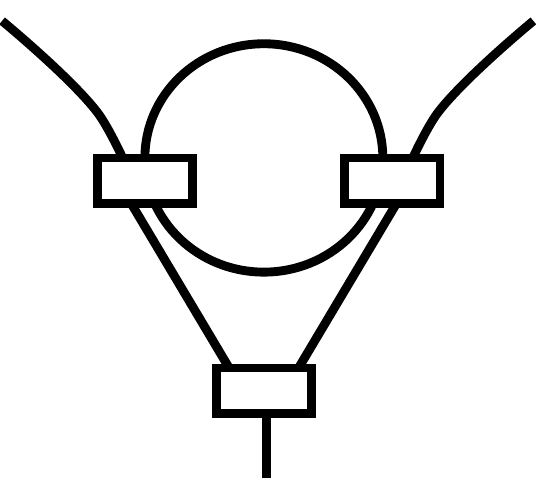}}%
    \put(0.9149574,0.64153913){\color[rgb]{0,0,0}\makebox(0,0)[lt]{\lineheight{1.25}\smash{\begin{tabular}[t]{l}$k$\end{tabular}}}}%
    \put(0.00392711,0.63974251){\color[rgb]{0,0,0}\makebox(0,0)[lt]{\lineheight{1.25}\smash{\begin{tabular}[t]{l}$k$\end{tabular}}}}%
    \put(0.45277613,0.42649505){\color[rgb]{0,0,0}\makebox(0,0)[lt]{\lineheight{1.25}\smash{\begin{tabular}[t]{l}$x$\end{tabular}}}}%
    \put(0.71497513,0.29854648){\color[rgb]{0,0,0}\makebox(0,0)[lt]{\lineheight{1.25}\smash{\begin{tabular}[t]{l}$y$\end{tabular}}}}%
    \put(0.18613307,0.29854648){\color[rgb]{0,0,0}\makebox(0,0)[lt]{\lineheight{1.25}\smash{\begin{tabular}[t]{l}$z$\end{tabular}}}}%
    \put(0.54610129,0.01930476){\color[rgb]{0,0,0}\makebox(0,0)[lt]{\lineheight{1.25}\smash{\begin{tabular}[t]{l}$a$\end{tabular}}}}%
    \put(0.3372308,0.86578479){\color[rgb]{0,0,0}\makebox(0,0)[lt]{\lineheight{1.25}\smash{\begin{tabular}[t]{l}$n-k$\end{tabular}}}}%
  \end{picture}%
\endgroup%
 
\caption{\label{f.skeinz} The skein element $\sk'$ in $K(\dd^2)$.}
\end{figure}

\begin{lem}[{\cite[Lemma 3.2]{lee2020jones}}] \label{l.dskein}
Let $\sk \in K(S^2)$ be a skein element. Suppose that there is a disk $\dd^2$ in $S^2$ so that the intersection of $\dd$ with $\sk$ is the skein  element $\sk' \in K(\dd^2)$ shown in Figure \ref{f.skeinz}.  If  $\frac{a}{2} - k > 0$, then $\langle \sk \rangle = 0$. 
\end{lem}

The final results in this section, that we will use  for the proof of Theorem \ref{t.main} to imply Theorem \ref{t.mm}, are Lemma \ref{l.adequatebound} and Corollary  \ref{c.ub} that provide upper bounds
on the degree of the Kauffman bracket of any skein element, in terms of the degree of the bracket of certain ``simpler" skein elements. Before we state the results we need to introduce some notation and terminology.

\begin{defn} \cite[Definition 2.3]{Russell}\footnote{The version of the definition given in this paper has ``points" rather than ``nodes" and we specialize to matching points on the boundary of a disk. }
A \textit{crossingless matching} on $2n$ points is an element in $TL_n$ that is  a collection of  $n$ disjoint arcs connecting the $2n$ points on the boundary of the disk
 $\dd^2$ defining $TL_n$.
 \end{defn}

The set of crossingless matchings forms a basis for $TL_n$  as a vector space over $\mathbb{C}(A)$.
Hence, every  $\mathcal{U}\in TL_n$ may be written in the form $\sum_uh(u) u$ where $h(u)\in\mathbb{C}(A)$ and $u$ runs over all crossingless matchings. In particular, the Jones-Wenzl projector has such an expansion from the following recursive formula \cite[Figure 13.6]{Lickorishbook}: 
\begin{figure} [H]
\begin{equation} \label{nton-1}
\def \svgwidth{.5\columnwidth}
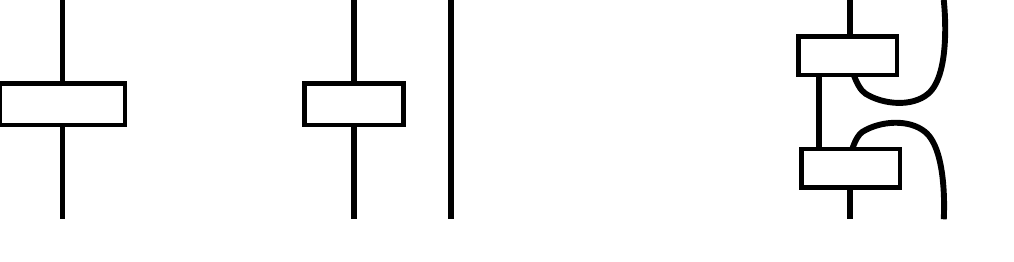
\end{equation} 
\caption{\label{f.nton-1} Recursive formula for the Jones-Wenzl projector.}
\end{figure} 

We note that, in contrast with crossingless matchings, crossingless skein elements in $TL_n$ may be empty.

\begin{defn} \label{d.barS} Given a crossingless skein element  $\sk \in K(S^2)$ (resp.  $T\in K(\dd^2)$) decorated with Jones-Wenzl projectors, we will use $\overline{\sk}$ (resp. $\overline{T}$) to denote the skein element obtained from $\sk$ (resp.  $T\in K(\dd^2)$) by replacing each projector with the identity element in the corresponding Temperley-Lieb algebra. 

Thus, $\overline{\sk}$ is a union of disjoint circles while $\overline{T}$ is a collection of disjoint arcs.

\end{defn} 

Next we prove the following lemma.

\begin{lem} \label{l.adequatebound} Suppose that
$\sk \in K(S^2)$ is a crossingless skein element decorated by a number of Jones-Wenzl projectors $\jw_n \in TL_n$
and, as above, let $\overline{\sk}$  denote the skein element resulting  from replacing each copy of  $\jw_n $ by $1_n \in TL_n$. Then, we have
 \[ \deg \langle \sk \rangle \leq \deg \langle \overline{\sk}\rangle.  \]
\end{lem}
\begin{proof} 
Suppose that the number of projectors contained in $\sk$ is $s$.
The proof will be by induction on $n$. 

Suppose $n+1=2$. Let $\sk$ be a skein element decorated by $s$ copies of $\jw_2$, and let  $\overline{\sk}$ denote the skein element resulting from replacing each copy of the projector with $1_2$.
Using \eqref{nton-1} we can expand each of the $s$ copies of $\jw_2$ in terms of the basis of crossingless matchings $\{1_2, e_1^2\}$, 
to write $\langle \sk \rangle $ into a sum of $2^s$ terms where each term is of the form
\[\deg \left(-\frac{\triangle_0}{\triangle_1} \right)^l \langle \overline{\sk}' \rangle \leq \deg   \langle \overline{\sk} \rangle, \] 
where $0\leq l\leq s$ and $\overline{\sk}'$ is a skein element obtained from $\overline{\sk}$ by replacing $l$ copies of the identity that were copies of  $\jw_2$ on the original $\sk$.
We can realize this process by a length $l$ sequence of skein elements 
$$\sk_0=\overline{\sk}\longrightarrow \sk_1 \longrightarrow \cdots \longrightarrow   \sk_l=\overline{\sk}',$$
such that at each step one identity $1_2$ replacing the projector is changed to $e_1^2$.
Since $\sk$ is crossingless, for $0\leq i \leq l$, $\sk_i$
is a collection of disjoint circles and in $\sk_{i+1}$ we merge or split two curves of $\sk_i$.
Thus
\[\deg \left(-\frac{\triangle_0}{\triangle_1} \right)\langle \sk_{i+1} \rangle= \deg \langle \sk_i \rangle -2\pm 2  \leq \deg \langle \sk_i \rangle,\] 
and the conclusion follows in this case.

Assuming inductively that the conclusion holds for $n>1$, let $\sk$ be a crossingless element decorated by $s$ copies of $\jw_{n+1}$.
Using \eqref{nton-1}  we can expand each of the $s$ copies in terms of the skein elements $\{T_0^n, T_1^n\}$ of  Figure \ref{f.nton-1}. Consider
the function $\alpha$ from the set of $s$ copies of    $\jw_{n+1}$       to  $\{T_0^n, T_1^n\}$.
Using  \eqref{nton-1} we write a sum of $2^s$ terns
\[\langle \sk \rangle = \sum_{\alpha}  f_{\alpha}  \langle \sk_{\alpha} \rangle ,   \] 
where $ f_{\alpha}:=\left(-\frac{\triangle_{n-1}}{\triangle_n} \right)^{\alpha_1} $ and $\alpha_1$ is the number of times $\alpha$ chooses $T_1^n$ for a copy of $\jw_{n+1}$ and $\sk_{\alpha} $ is the skein element obtained from $\sk$ by applying $\alpha$.
 By the induction hypothesis, we have 
\begin{equation}
\deg \langle \sk_\alpha  \rangle \leq \deg \langle \overline{\sk_{\alpha}} \rangle.  
\label{induction}
\end{equation}

Note that each
$\overline{\sk_{\alpha}}$ comes from $\sk_{\alpha}$ by replacing each  $T_0 ^n$, $T_1 ^n$ assigned by $\alpha$
by $\overline{T_1^n}$,  $\overline{T_0^n}$, respectively.  Also $ \overline{\sk}$ comes from the function $\alpha^0$ that chooses $T_0^n$ for each of the $s$ copies of the projector.
For any $\sk_{\alpha}$ we can take a sequence of $\alpha_1$ skein elements that starts from $ \overline{\sk}$ and ends at $\overline{\sk_{\alpha}}$, where each step switches
a single copy of $\overline{T_0^n}$ to  $\overline{T_1^n}$.
The replacement of a copy of $\overline{T_0^n}$ by $\overline{T_1^n}$ can merge or split two circles in the previous skein element of sequence.
As earlier we conclude that for $\alpha\neq \alpha^0$, 
$\deg f_{\alpha} \langle \overline{\sk_{\alpha}} \rangle \leq  \langle \overline{\sk} \rangle$.
Combining this with \eqref{induction} we get
\[ \deg f_\alpha \langle \sk_\alpha \rangle  \leq \deg f_\alpha \langle \overline{\sk_\alpha} \rangle\leq \deg \langle \overline{\sk} \rangle. \] 

Then since 
\[\deg \langle \sk \rangle \leq \max \{ \deg f_\alpha \langle \sk_\alpha \rangle\},\]
we have  $\deg \langle \sk \rangle  \leq \deg \langle \overline{\sk} \rangle$.

\end{proof}

A consequence of Lemma \ref{l.adequatebound} is the following. 
\begin{cor} \label{c.ub} Suppose that $\sk \in K(S^2)$ is a skein element with crossings which contains a number of  Jones-Wenzl projectors $\jw_n \in TL_n$. Let $\overline{\sk_A}$ be the skein element obtained from $\sk$  by first applying the all-$A$ Kauffman state, then replacing each of the projectors in $\sk_A$ by the identity $1_n$. Then, we have 
\[\deg \langle  \sk \rangle \leq \deg \left( A^{\sgn(\sigma_A)} \langle \overline{\sk_A} \rangle \right). \] 
\end{cor} 
\begin{proof}
 If $\sk$ does not contain any Jones-Wenzl projectors, then $\overline{\sk} = \sk$ and  we have a link diagram in $K(S^2)$ for which it is well known \cite{Lickorishbook} that 
 \[\deg \langle \sk \rangle \leq \max_{\sigma} \{ \deg \left(  A^{\sgn(\sigma)}  \langle \sk_\sigma \rangle  \right) \}  \leq \deg \left( A^{\sgn(\sigma_A)}  \langle \sk_A\rangle \right).   \]
 
 If $\sk$ contains Jones-Wenzl projectors, then first we apply Kauffman states to all the crossings of $\sk$ to expand $\sk = \sum_{\sigma} A^{\sgn(\sigma)} \langle \sk_\sigma \rangle$ and then apply Lemma \ref{l.adequatebound} to each term to get 
\[\deg \langle \sk_\sigma \rangle \leq \deg \langle \overline{\sk_\sigma} \rangle. \]
Hence we have

\[  \deg \langle \sk \rangle \leq \max_\sigma \{ \deg  \left(  A^{\sgn(\sigma)} \langle \sk_\sigma \rangle \right) \}  \leq \max_\sigma \{ \deg  \left(  A^{\sgn(\sigma)} \langle \overline{\sk_\sigma} \rangle \right) \}.  \]  
 Now for each state we have  $\overline{\sk_\sigma} = (\overline{\sk})_\sigma$, where $(\overline{\sk})_\sigma$ comes from applying the Kauffman state $\sigma$ to the link diagram $\overline{\sk}$.  Thus we can write
\[ \deg  \left(  A^{\sgn(\sigma)} \langle \overline{\sk_\sigma} \rangle \right) \leq \deg \left( A^{\sgn(\sigma)}\langle (\overline{\sk})_\sigma \rangle\right)  \leq \deg \left(  A^{\sgn(\sigma_A)} \langle  (\overline{\sk})_A \rangle \right)  = \deg \left( A^{\sgn(\sigma_A)}  \langle \overline{\sk_A} \rangle \right).\] 
Now the conclusion follows by combining the last inequality with the preceding one. 
\end{proof}

\section{Refined bounds on the quadratic growth rate of the degree }
In this section we will prove Theorem \ref{t.mm}.
It follows from Corollary \ref{c.ub} that there is an upper bound $H_n(D)$ for $\deg \langle \Dj^n \rangle$: 
\[ \deg \langle \Dj^n \rangle \leq H_n(D) := c(D) n^2 + 2v_A(D)n.  \] 
This is obtained by directly calculating the number of circles in the all-$A$ state on $\overline{\Dj^n}$ and the number of crossings on which the $A$-resolution is chosen. 

\begin{defn} For a diagram $D=D(K)$ recall that $\writhe(D)$ is the writhe of $D$ and define 
\begin{equation} \label{e.h_n} h_n(D):= - \frac{H_n(D)}{4} + \writhe(D)\frac{n^2-1}{4}. \end{equation}  By Definition \ref{d.cjp}, $h_n(D)$ is a lower bound of $\deg J_K(n)$.
\end{defn}

\begin{thm} \label{t.main} Suppose a knot diagram $D$ is not $A$-adequate, then 
\[ d_-[J_K(n)] \geq  h_n(D) + p_-(D)n^2 + O(n), \]
 for a constant $p_{-}(D)>0$ which depends on the diagram $D$.  
\end{thm} 

Before proving Theorem \ref{t.main} we show it implies Theorem \ref{t.mm}. 

\subsection*{Proof of Theorem \ref{t.mm}}
\begin{proof} 
First, the statement of Theorem \ref{t.main} directly gives Theorem \ref{t.mm} \eqref{t.mma} since $h_n(D) = \frac{-c(D)+\writhe(D)}{4}n^2 + O(n)$. Secondly, if $D$ is not $B$-adequate, then taking the mirror image of $D$ gives a diagram $D^*$ that is not $A$-adequate. Applying Theorem \ref{t.main} to $K^* = D^*$ and using the fact that $d_+[ J_K(n)] = -d_-[ J_{K^*}(n)]$ gives Theorem \ref{t.mm} \eqref{t.mmb}.  Part (c) follows from combining parts \eqref{t.mma} and \eqref{t.mmb} to get $s_2\leq 2c_+(D) -4p_+(D)$ and  $s_2^*\geq -2c_-(D) + 4p_-(D)$. 
This implies \[ d[J_K(n)] \leq  (2c(D)- 4p(D))n^2  + O(n), \]  
where $p(D): = p_+(D) + p_-(D)$. 
\end{proof} 

\subsection{State graphs and through strands} It turns out that information about contributions of individual Kaufman states on $\Dj^n$ 
 to $\deg \langle \Dj^n \rangle$ is encoded in  certain walks on the all-$A$ state graph $\mathbb{G}_A$ of $D$. In this subsection, 
 we set up this correspondence and prove an auxiliary lemma that will be used in the proof of Theorem \ref{t.main}.

 Let $u, v\in TL_{2n}$ be two crossingless matchings embedded in two copies of the disk $ \dd^2\subset S^2$ as shown in the left panel of Figure \ref{f.taddition}. The points of each of $u,v$ on the boundary  $\partial \dd^2$ of the disk
 are separated into four groups that can be labeled as  northwest (NW), northeast (NE), southwest (SW) and southeast (SE).
 Generalize the notions of addition and  numerator closure of $(2n, 2n)$-tangles,
we define $N(u + v)$ to be the skein element in $K(S^2)$ obtained by joining
 the NE (resp. SE) points  of $u$ on $\partial \dd^2$  to the NW (resp. SW)  points of $v$ by parallel arcs (right  panel of Figure \ref{f.taddition}).
 Similarly we join the NW (resp. SW) points of $u$ on $\partial \dd^2$   to the NE (resp. SE) points of $v$.
 \begin{figure}[H]
\def \svgwidth{.5\columnwidth} 
%% Accompanies image file 'taddition.pdf' (pdf, eps, ps)
\begingroup%
  \makeatletter%
  \providecommand\color[2][]{%
    \errmessage{(Inkscape) Color is used for the text in Inkscape, but the package 'color.sty' is not loaded}%
    \renewcommand\color[2][]{}%
  }%
  \providecommand\transparent[1]{%
    \errmessage{(Inkscape) Transparency is used (non-zero) for the text in Inkscape, but the package 'transparent.sty' is not loaded}%
    \renewcommand\transparent[1]{}%
  }%
  \providecommand\rotatebox[2]{#2}%
  \newcommand*\fsize{\dimexpr\f@size pt\relax}%
  \newcommand*\lineheight[1]{\fontsize{\fsize}{#1\fsize}\selectfont}%
  \ifx\svgwidth\undefined%
    \setlength{\unitlength}{727.92425681bp}%
    \ifx\svgscale\undefined%
      \relax%
    \else%
      \setlength{\unitlength}{\unitlength * \real{\svgscale}}%
    \fi%
  \else%
    \setlength{\unitlength}{\svgwidth}%
  \fi%
  \global\let\svgwidth\undefined%
  \global\let\svgscale\undefined%
  \makeatother%
  \begin{picture}(1,0.47134696)%
    \lineheight{1}%
    \setlength\tabcolsep{0pt}%
    \put(0,0){\includegraphics[width=\unitlength,page=1]{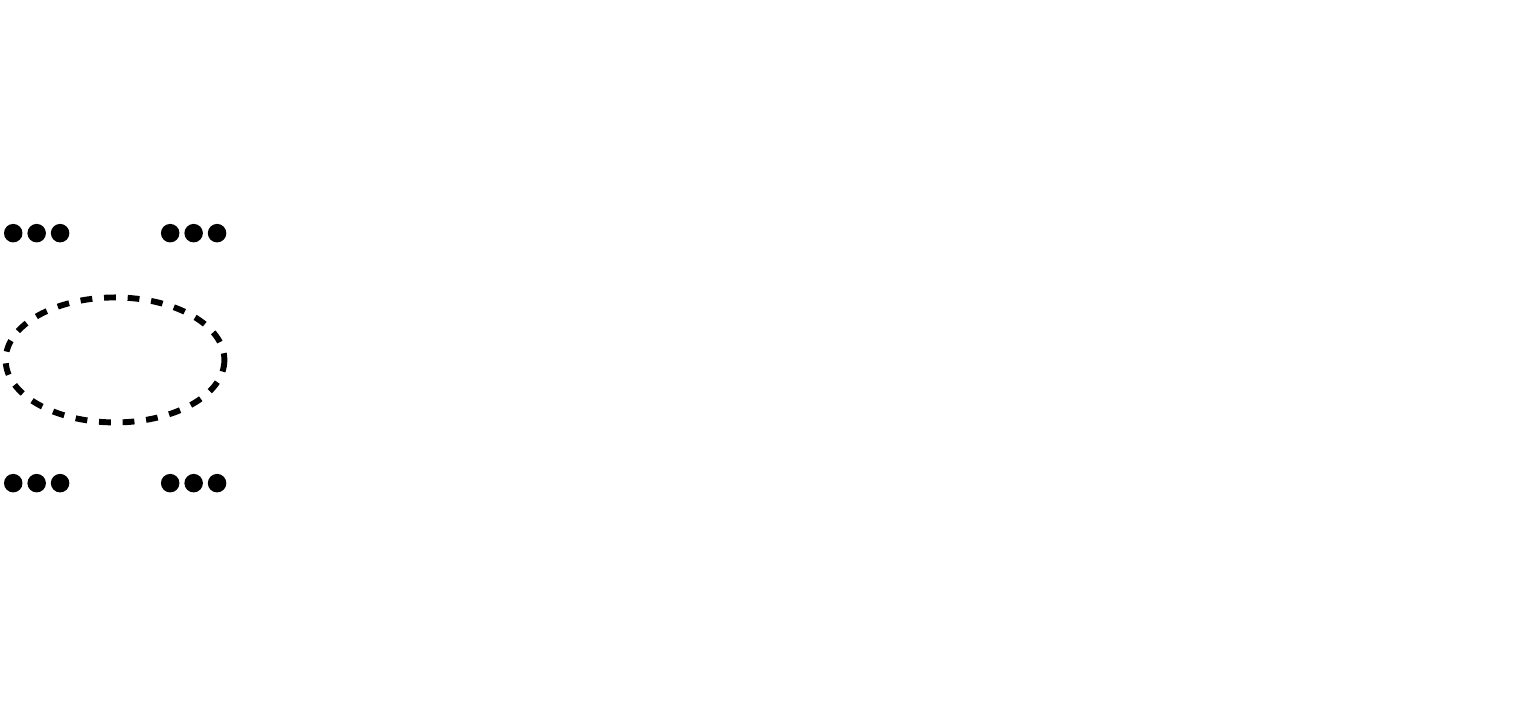}}%
    \put(0.05524124,0.23399691){\makebox(0,0)[lt]{\lineheight{1.25}\smash{\begin{tabular}[t]{l}$u$\end{tabular}}}}%
    \put(0,0){\includegraphics[width=\unitlength,page=2]{taddition.pdf}}%
    \put(0.3232911,0.23399691){\makebox(0,0)[lt]{\lineheight{1.25}\smash{\begin{tabular}[t]{l}$v$\end{tabular}}}}%
    \put(0,0){\includegraphics[width=\unitlength,page=3]{taddition.pdf}}%
    \put(0.63714291,0.23527464){\makebox(0,0)[lt]{\lineheight{1.25}\smash{\begin{tabular}[t]{l}$u$\end{tabular}}}}%
    \put(0,0){\includegraphics[width=\unitlength,page=4]{taddition.pdf}}%
    \put(0.90519284,0.23527464){\makebox(0,0)[lt]{\lineheight{1.25}\smash{\begin{tabular}[t]{l}$v$\end{tabular}}}}%
    \put(0,0){\includegraphics[width=\unitlength,page=5]{taddition.pdf}}%
  \end{picture}%
\endgroup%
  
\caption{\label{f.taddition} $N(u+v)$ for $u, v \in TL_{6}$}
\end{figure} 

For two skein elements $\mathcal{U}= \sum h(u)u $ and $\mathcal{V} = \sum g(v) v$, where $u$ and $v$ are crossingless matchings in $TL_{2n}$, we define an addition operation $N(\mathcal{U} +  \mathcal{V})$ to be 
\[ N(\mathcal{U} +  \mathcal{V}) := \sum h(u) g(v) N(u +  v).  \]

\begin{defn}
Let $\sk$ be a crossingless skein element in $TL_n$ which does not contain a Jones-Wenzl projector. A \textit{through strand} of $\sk$ is an arc in $\sk$ with one endpoint on the top of the disk and the other endpoint on the bottom of the disk. 
\end{defn} 

Given a skein element $\sk \in K(S^2)$ with crossings and a Kauffman state $\sigma=\sigma(\sk)$, there is a sequence of states
$s:=\{ \sigma_1 = \sigma_A\rightarrow\ldots \rightarrow \sigma_f = \sigma \}$ 
 from the all-$A$ state $\sigma_A$ to $\sigma$ such that, for $1\leq i \leq f-1$, $\sigma_{i+1}$  is obtained from $\sigma_i$
 by replacing the $A$-resolution at a single crossing with  the $B$-resolution.
 
 Given a single crossing $y$ on  $\sk \in K(S^2)$, it will be convenient for us to view it as a skein element  in $TL_2$. Then the $n$-cable of $y $ is an element in $TL_{2n}$ which we will denote by $y^n$.
 If $y^n$ is part of a skein element $\sk \in K(S^2)$ and $\sigma$ is a state on $ \sk$
 we will use $\sigma|y^n$ to
 denote the restriction of $\sigma$ on $y^n$ and we will use $\sigma(y^n)$ to denote the element of $TL_{2n}$ resulting from $\sigma|y^n$.

\begin{lem}{{\cite[Lemma 3.7(b)]{lee2020jones}}} \label{l.xoriented} Let  $y$ be a single crossing $\vcenter{\hbox{\includegraphics[scale=.1]{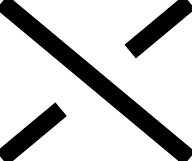}}}$ viewed as an element of $TL_2$. If the skein element $\sigma(y^n)$ has $2k$ through strands, then the sequence of states from $\sigma_A|y^n$  to $\sigma|y^n$ contains a subsequence  of length $k^2$. 
\end{lem} 

\begin{proof} Viewing the disk $\dd^2$ as $[0, 1]\times [0, 1]$, consider $\sigma_A(y^n)$ in the disk and isotope $y^n$ so that each set of dashed segments between the same pair of state arcs in $\sigma_A(y^n)$ are in the same horizontal strip $[0, 1] \times [y_{i}, y_{i'}]$ without any overlap, see Figure \ref{f.xnsquare}.
\begin{figure}[H]
\def \svgwidth{.3\columnwidth}
%% Accompanies image file 'xnsquare.pdf' (pdf, eps, ps)
\begingroup%
  \makeatletter%
  \providecommand\color[2][]{%
    \errmessage{(Inkscape) Color is used for the text in Inkscape, but the package 'color.sty' is not loaded}%
    \renewcommand\color[2][]{}%
  }%
  \providecommand\transparent[1]{%
    \errmessage{(Inkscape) Transparency is used (non-zero) for the text in Inkscape, but the package 'transparent.sty' is not loaded}%
    \renewcommand\transparent[1]{}%
  }%
  \providecommand\rotatebox[2]{#2}%
  \newcommand*\fsize{\dimexpr\f@size pt\relax}%
  \newcommand*\lineheight[1]{\fontsize{\fsize}{#1\fsize}\selectfont}%
  \ifx\svgwidth\undefined%
    \setlength{\unitlength}{444.79684905bp}%
    \ifx\svgscale\undefined%
      \relax%
    \else%
      \setlength{\unitlength}{\unitlength * \real{\svgscale}}%
    \fi%
  \else%
    \setlength{\unitlength}{\svgwidth}%
  \fi%
  \global\let\svgwidth\undefined%
  \global\let\svgscale\undefined%
  \makeatother%
  \begin{picture}(1,0.4568747)%
    \lineheight{1}%
    \setlength\tabcolsep{0pt}%
    \put(0,0){\includegraphics[width=\unitlength,page=1]{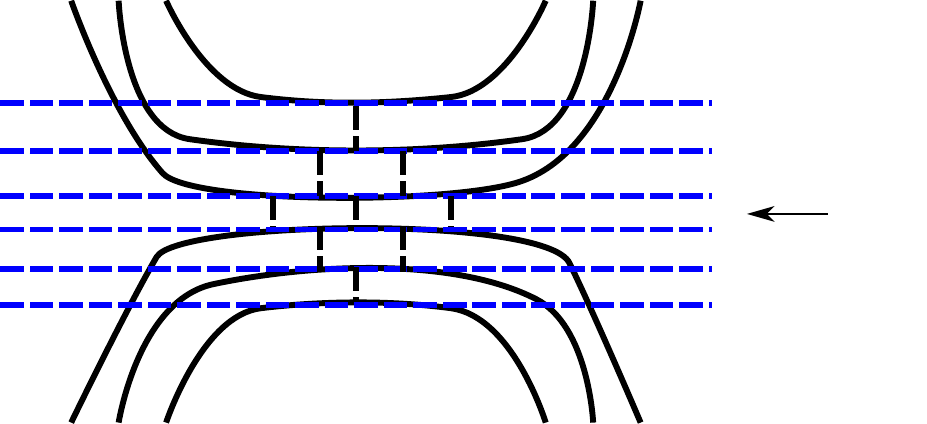}}%
    \put(0.910528,0.21596328){\color[rgb]{0,0,0}\makebox(0,0)[lt]{\lineheight{1.25}\smash{\begin{tabular}[t]{l}middle\end{tabular}}}}%
  \end{picture}%
\endgroup%

\caption{\label{f.xnsquare} The skein $\sigma_A(y^n)$ with  horizontal strips between state arcs.} 
\end{figure}  

We consider $\sigma(y^n)$ in the same disk. The intersection of any horizontal segment $[0, 1] \times y_0$ with $\sigma(y^n)$ has to contain $2k$ points since $\sigma(y^n)$ is assumed to have $2k$ through strands. For each horizontal strip that contains $0 \leq \ell \leq  n$ crossings, there are $2(n-\ell)$ state arcs providing part of the requisite $2k$ intersections.   We add up the remaining number of crossings necessary: $k, 2(k-1), 2(k-2), \ldots, 2$ to get $k^2$ crossings in $c_B(\sigma)$, where recall $c_B(\sigma)$ is the number of crossings on which $\sigma$ chooses the $B$-resolution. 
\end{proof}

\begin{remark} \label{r.throughe}
 Let  $y$ be a single crossing $\vcenter{\hbox{\includegraphics[scale=.1]{xoriented.pdf}}}$ viewed as an element of $TL_2$ and let $\sigma$ be a Kauffman state on $y^n$. Then a through strand of $\sigma(y^n)$ runs in a direction parallel on the projection plane to the dashed segment of Figure \ref{f.kstate}  on  $\sigma_A(y)$.
If $y$ is part of a skein element $\sk\in K(S^2)$ then
this dashed segment gives an edge
 $e_y$ of  ${\mathbb G}_A(\sk)$ of $D$. See Figure \ref{f.throughe}.

  \begin{figure}[H]
\begin{center} 
\def \svgwidth{.8\columnwidth}
%% Creator: Inkscape 1.0beta1 (32d4812, 2019-09-19), www.inkscape.org
%% Accompanies image file 'throughe.pdf' (pdf, eps, ps)
\begingroup%
  \makeatletter%
  \providecommand\color[2][]{%
    \errmessage{(Inkscape) Color is used for the text in Inkscape, but the package 'color.sty' is not loaded}%
    \renewcommand\color[2][]{}%
  }%
  \providecommand\transparent[1]{%
    \errmessage{(Inkscape) Transparency is used (non-zero) for the text in Inkscape, but the package 'transparent.sty' is not loaded}%
    \renewcommand\transparent[1]{}%
  }%
  \providecommand\rotatebox[2]{#2}%
  \newcommand*\fsize{\dimexpr\f@size pt\relax}%
  \newcommand*\lineheight[1]{\fontsize{\fsize}{#1\fsize}\selectfont}%
  \ifx\svgwidth\undefined%
    \setlength{\unitlength}{641.09536959bp}%
    \ifx\svgscale\undefined%
      \relax%
    \else%
      \setlength{\unitlength}{\unitlength * \real{\svgscale}}%
    \fi%
  \else%
    \setlength{\unitlength}{\svgwidth}%
  \fi%
  \global\let\svgwidth\undefined%
  \global\let\svgscale\undefined%
  \makeatother%
  \begin{picture}(1,0.15491641)%
    \lineheight{1}%
    \setlength\tabcolsep{0pt}%
    \put(0,0){\includegraphics[width=\unitlength,page=1]{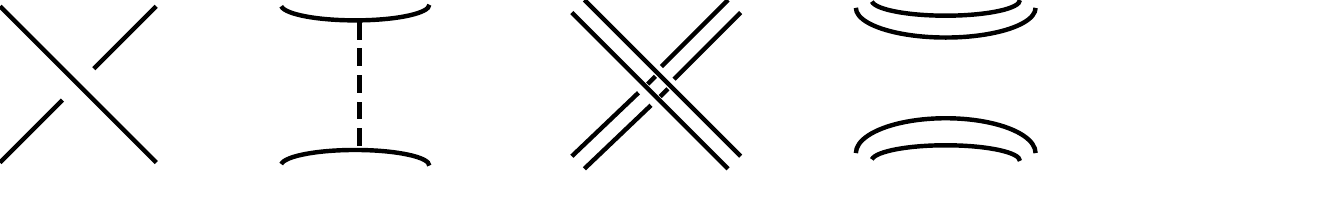}}%
    \put(0.29246819,0.05800199){\makebox(0,0)[lt]{\lineheight{1.25}\smash{\begin{tabular}[t]{l}$e_y$\end{tabular}}}}%
    \put(0.04211543,0.00983953){\makebox(0,0)[lt]{\lineheight{1.25}\smash{\begin{tabular}[t]{l}$y$\end{tabular}}}}%
    \put(0.46629472,0.00404894){\makebox(0,0)[lt]{\lineheight{1.25}\smash{\begin{tabular}[t]{l}$y^2$\end{tabular}}}}%
    \put(0.64472755,0.00322172){\makebox(0,0)[lt]{\lineheight{1.25}\smash{\begin{tabular}[t]{l}$\sigma_A(y^2)$\end{tabular}}}}%
    \put(0,0){\includegraphics[width=\unitlength,page=2]{throughe.pdf}}%
    \put(0.86826635,0.00322171){\makebox(0,0)[lt]{\lineheight{1.25}\smash{\begin{tabular}[t]{l}$\sigma(y^2)$\end{tabular}}}}%
    \put(0,0){\includegraphics[width=\unitlength,page=3]{throughe.pdf}}%
    \put(0.21655415,0.00907107){\makebox(0,0)[lt]{\lineheight{1.25}\smash{\begin{tabular}[t]{l}$\sigma_A(y)$\end{tabular}}}}%
  \end{picture}%
\endgroup%

\caption{\label{f.throughe} The through strands in $\sigma(y^2)$ come from changing the $A$-resolution to the $B$-resolution on the leftmost crossing in $y^2$.}
\end{center} 
\end{figure}

More generally, a $(2,2)$-tangle ${\mathcal T}$ that is part of a skein element  $\sk \in K(S^2)$, where ${\mathcal T}$ is with crossings and without Jones-Wenzl projectors,
 is viewed as an element in $ TL_2.$  On the  $n$-blackboard cable of $\sk^n$, we have the $n$-blackboard cable ${\mathcal T}^n$,
 which is an element in $ TL_{2n}$. 
A  Kauffman state $\sigma$ on $\sk^n$ restricted to ${\mathcal T}^n$  produces a crossingless skein element
$\sigma({\mathcal T}^n)$ without projectors, and it makes sense to talk about its through strands.

\begin{defn} \label{d.pass} With the notation and setting as above, we will say that a through strand of $\sigma({\mathcal T}^n)$  \emph{passes through} a crossing $y$ of $\sk$,
if it contains some state arcs of  $\sigma | y^n$.
\end{defn} 

Now recall that for a graph with edge set $E$ and vertex set $V$, where $v_0(e), v_1(e)$ denote the vertices of an edge $e\in E$,  a \emph{walk} is a sequence of edges $e_1, \ldots, e_m \in E$ such that $v_1(e_i) = v_0(e_{i+1})$. A walk is \emph{closed} if $v_0(e_1) = v_1(e_m)$.

If the skein element $\sk$ is a knot diagram $D=D(K)$, then for every through strand, say $\alpha$, of  $\sigma({\mathcal T}^n)$ in $\Dj^n$
we obtain a walk of edges on the all-$A$ state graph ${\mathbb G}_A$ of $D$. Namely, the walk consists of edges $\{e_y\}$ corresponding to all crossings $y\in \sk$ that $\alpha$ passes through.

The observation that we can view through strands of skein elements resulting from applying Kauffman states to tangles as walks on  ${\mathbb G}_A$ will be used in the proof of Theorem \ref{t.main}.
\end{remark}

\subsection{Proof of Theorem \ref{t.main}} Before we embark on the formal proof of the theorem we give a brief outline of it.
Starting with a knot diagram $D=D(K)$ that is not $A$-adequate, we fix a crossing $x$ that corresponds to a one-edged loop  $e$ of the state graph
${\mathbb G}_A={\mathbb G}_A(D)$. For fixed $n$ we view ${\Dj}^n$ as a sum of two $(2,2)$-tangles, one consisting of $x$ and the complimentary tangle ${\mathcal T}$.
Using fusion rules (Lemma \ref{l.Twistdegree}) on the fixed crossing $x$, we write
${\Dj}^n $ as a  sum of skein elements $\sk_{\sigma}(a)$ parametrized by pairs consisting of the fusion parameter,  $a$ of $x$, and Kauffman states $\sigma$.
We write $\sk_{\sigma}(a)=N(\mathcal{I}(a, -1, n) + \mathcal{T}_\sigma),$   where $\mathcal{I}(a, -1, n)$ are the skein elements of Definition \ref{notation} and
$ \mathcal{T}_\sigma \in  TL_{2n}$ are skein elements resulting from applying $\sigma | {\mathcal T}$.
Using  Lemmas \ref{l.dskein}, \ref{l.adequatebound}  we are able to  isolate the states $\sigma$ that contribute to the maximum degree $\deg \langle \Dj^n \rangle$.
To estimate the degree  contributions of such a state $\sigma$, we view the through strands of the skein elements $\mathcal{T}_\sigma$ as closed  walks on 
${\mathbb G}_A$ starting and ending at the vertex containing the one-edged loop $e$. 
This allows us to relate the number of through strands of $\mathcal{T}_\sigma$ 
to the number of crossings on which $\sigma$ chooses the $B$-resolution. This, in turn, combined with Lemma \ref{l.xoriented}, allows us to estimate that the drop  of the degree $\deg \langle \Dj^n \rangle$
from the potential maximum $H_n(D)$, grows quadratically in $n$. See  \eqref{e.maininequality} below for the precise statement.

\begin{proof} 

Given a knot diagram $D$ of $K$, it suffices to show that if $D$ is not $A$-adequate, then  
\begin{equation} \label{e.maininequality}
\deg \langle \Dj^n \rangle \leq  H_n(D)  - \rho n^2 + O(n), 
\end{equation} 
for a constant $\rho > 0$ depending on the diagram. The statement of the theorem then follows from \eqref{e.maininequality} passing to $h_n(D)$ using Equation \eqref{e.h_n} with $p_-(D)= \frac{1}{4}\rho$. 

\begin{figure}[H]
\begin{center} 
\def \svgwidth{.8\columnwidth}
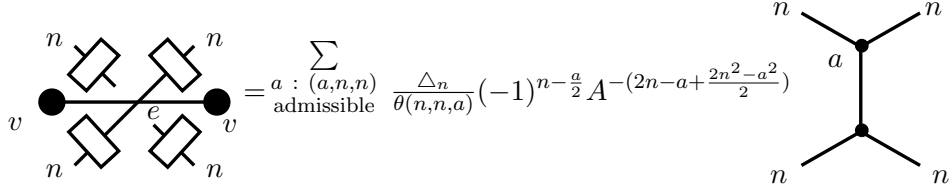
\caption{\label{f.fusingonepos} The crossing corresponding to $e$  framed with  Jones-Wenzl projectors and fused along a direction transverse to $e$. 
}
\end{center} 
\end{figure}

Since $D$ is not $A$-adequate,  $\mathbb{G}_A$  has  one-edged loops. Fix a crossing   $x$ whose $A$-resolution gives a one-edged loop $e$ with vertex $v$ in $\mathbb{G}_A$.
Take the $n$-blackboard cable of $D$ and decorate with a Jones-Wenzl projector  as in Definition \ref{d.cjp}.  
Double and slide the projectors along the $n$-cable of the knot using its defining properties from Definition \ref{d.jw}, so that there are four projectors framing $x^n$. Let $\Dj^n$ be the resulting diagram. 

 To compute $\langle \Dj^n \rangle$ first apply Lemma \ref{l.Twistdegree} (with $r=-1$)  for the twist region $T$ consisting of the crossing $x$. See Figure \ref{f.fusingonepos} for the illustration, where since $e$ is a one-edged loop, its endpoints are on a single vertex $v$. Note that we are choosing a direction transverse to $e$ for the fusion. 
We have
\begin{equation} \label{e.expansion}
\langle \Dj^n \rangle 
=  \sum_{\substack{a: (a, n, n) \ \text{admissible,} \\ 0\leq a\leq 2n}} \frac{\triangle_n}{\theta(n, n, a )} (-1)^{n-\frac{a}{2}} A^{-(2n-a+ n^2-\frac{a^2}{2})} \langle \sk(a) \rangle,
\end{equation}
where  $\sk(a)$ denotes the skein element corresponding to the parameter $a$ in the sum of Figure \ref{f.fusingonepos}.
Next,  we expand the sum of \eqref{e.expansion} over all Kauffman states $\sigma$ on the crossings of each $\sk(a)$ to get  

\begin{align} \label{e.expansion2}
\langle \Dj^n \rangle &= \sum_{\substack{\sigma \text{ Kauffman state on $\sk(a)$}, \\  a: (a, n, n) \ \text{admissible,} \\ 0\leq a \leq 2n}} I(a, -1, n)   A^{\sgn(\sigma)}  \langle \mathcal{S}_{\sigma}(a) \rangle,
\end{align} 
where  $\sk_\sigma(a)$ denotes the skein element resulting from applying a Kauffman state $\sigma$ to the crossings of $\sk(a)$. Finally, as in Lemma \ref{l.Twistdegree}, $I(a, -1, n): = \frac{\triangle_n}{\theta(n, n, a )} (-1)^{n-\frac{a}{2}}  A^{-(2n-a+ \frac{2n^2-a^2}{2})}$. See \cite[Proposition 5.1]{Lickorishbook}\footnote{Proposition 5.1 in \cite{Lickorishbook} extends to the case of a skein element with crossings and decorated with projectors with identical proof. }. 
As defined earlier in Definition \ref{d.stategraphsign},  $\sgn(\sigma)$ denotes the number of crossings of $\Dj^n$ for which $\sigma$ assigns the $A$-resolution minus the number of crossings for which $\sigma$ assigns the $B$-resolution.

Note that  $\sk(a) = N(\mathcal{I}(a, -1, n) +  \mathcal{T}$) where $\mathcal{I}(a, -1, n)$ is the skein element of Definition \ref{notation} for $r=-1$, and $ \mathcal{T}$ is the diagram $\Dj^n$ with $T^n$ removed.   
Similarly, $$\sk_{\sigma}(a)=N(\mathcal{I}(a, -1, n) + \mathcal{T}_\sigma),$$ 
where $\mathcal{T}_\sigma \in TL_{2n}$ denotes the complement of $\mathcal{I}(a, -1, n)$ in  $\sk_\sigma(a)$. 

The through strands of $\mathcal{T}_\sigma$ are the strands that run in a direction parallel, on the projection plane, to the edge of $\mathcal{I}(a, -1, n)$ labeled by $a$.
Since $\mathcal{T}_\sigma\in TL_{2n}$, its number of through strands is even. Let $k_{\sigma}$  denote half of this number.

\begin{figure}[H]
\def \svgwidth{.3\columnwidth}
%% Accompanies image file 'skeinJ.pdf' (pdf, eps, ps)
\begingroup%
  \makeatletter%
  \providecommand\color[2][]{%
    \errmessage{(Inkscape) Color is used for the text in Inkscape, but the package 'color.sty' is not loaded}%
    \renewcommand\color[2][]{}%
  }%
  \providecommand\transparent[1]{%
    \errmessage{(Inkscape) Transparency is used (non-zero) for the text in Inkscape, but the package 'transparent.sty' is not loaded}%
    \renewcommand\transparent[1]{}%
  }%
  \providecommand\rotatebox[2]{#2}%
  \newcommand*\fsize{\dimexpr\f@size pt\relax}%
  \newcommand*\lineheight[1]{\fontsize{\fsize}{#1\fsize}\selectfont}%
  \ifx\svgwidth\undefined%
    \setlength{\unitlength}{558.80950207bp}%
    \ifx\svgscale\undefined%
      \relax%
    \else%
      \setlength{\unitlength}{\unitlength * \real{\svgscale}}%
    \fi%
  \else%
    \setlength{\unitlength}{\svgwidth}%
  \fi%
  \global\let\svgwidth\undefined%
  \global\let\svgscale\undefined%
  \makeatother%
  \begin{picture}(1,0.72877772)%
    \lineheight{1}%
    \setlength\tabcolsep{0pt}%
    \put(0,0){\includegraphics[width=\unitlength,page=1]{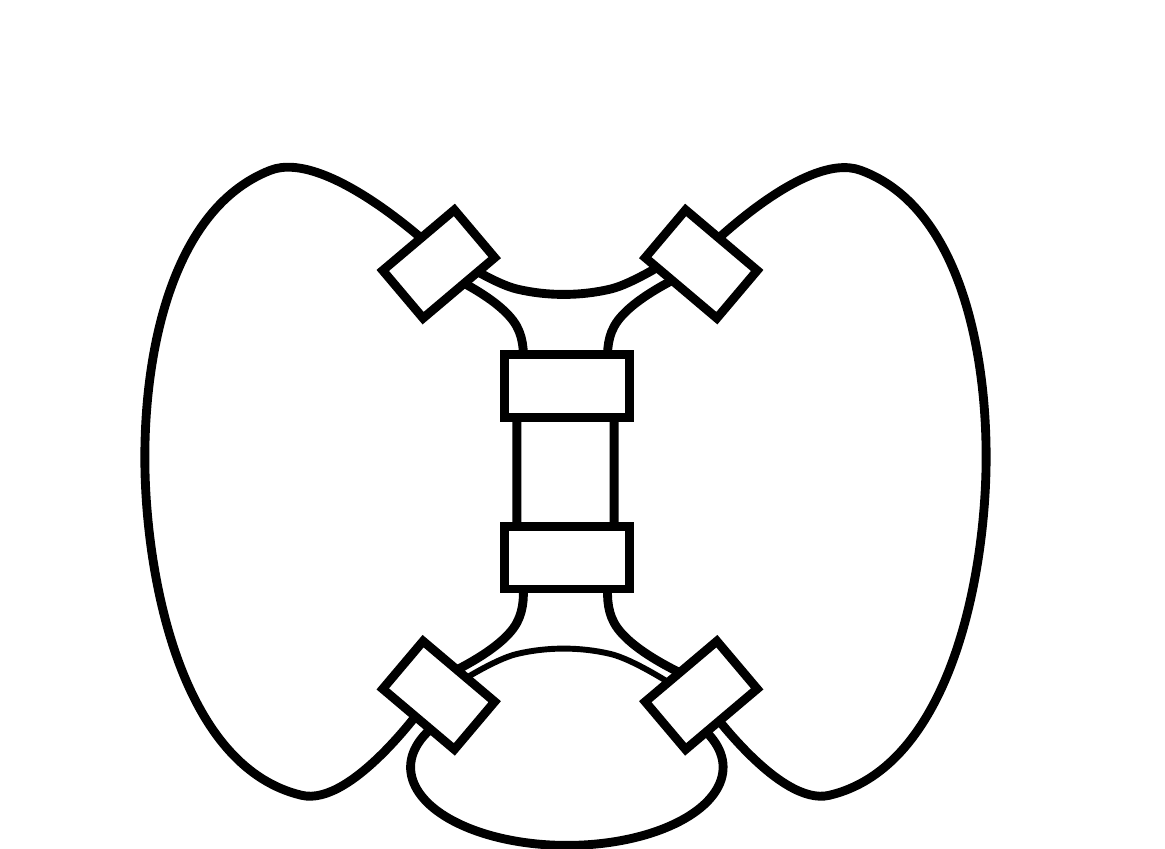}}%
    \put(0.30438242,0.30333238){\color[rgb]{0,0,0}\makebox(0,0)[lt]{\lineheight{1.25}\smash{\begin{tabular}[t]{l}$\frac{a}{2}$\end{tabular}}}}%
    \put(0.63186452,0.30333238){\color[rgb]{0,0,0}\makebox(0,0)[lt]{\lineheight{1.25}\smash{\begin{tabular}[t]{l}$\frac{a}{2}$\end{tabular}}}}%
    \put(-0.00162525,0.34091231){\color[rgb]{0,0,0}\makebox(0,0)[lt]{\lineheight{1.25}\smash{\begin{tabular}[t]{l}$ k_\sigma$\end{tabular}}}}%
    \put(0.86616367,0.33899497){\color[rgb]{0,0,0}\makebox(0,0)[lt]{\lineheight{1.25}\smash{\begin{tabular}[t]{l}$ k_\sigma$\end{tabular}}}}%
    \put(0,0){\includegraphics[width=\unitlength,page=2]{skeinJ.pdf}}%
  \end{picture}%
\endgroup%

\caption{\label{3.11}The skein element $\mathcal{J}_\sigma(a)$. } 
\end{figure}

Since the component, say $\mathcal{J}_\sigma(a)$,  of $\sk_\sigma(a)$ decorated with Jones-Wenzl projectors contains $\mathcal{I}(a, -1, n)$, Lemma \ref{l.dskein} applies to conclude that $\langle J_{\sigma}(a) \rangle$ and therefore $\langle \sk_\sigma(a) \rangle$ is 0 unless $\frac{a}{2} \leq k_{\sigma}$. See Figure \ref{3.11} for an illustration, where the dashed circle indicates the skein element of a disk where Lemma \ref{l.dskein} applies.

Thus, we can rewrite the sum in Equation (\ref{e.expansion2}) as
\begin{align}
&\langle \Dj^n \rangle = \sum_{\substack{\sigma \text{ Kauffman state on $\sk(a)$}, \\  a: (a, n, n) \ \text{admissible,} \\ 0\leq a \leq 2n,  \frac{a}{2} \leq k_{\sigma} }} I(a, -1, n) A^{\sgn(\sigma)}  \langle \mathcal{S}_{\sigma}(a) \rangle.  \label{e.dn}
\end{align}

Now we consider the degree of each term in  Equation \eqref{e.dn}.
We have
$$\deg \left( I(a, -1, n) A^{\sgn(\sigma)}  \langle \sk_{\sigma}(a) \rangle \right) = d(a, -1, n) + \sgn(\sigma) + \deg \langle \sk_{\sigma}(a) \rangle,$$
where, by Lemma \ref{l.Twistdegree}, $d(a, -1, n) = \deg I(a, -1, n)=-4n + 2a - n^2 + \frac{a^2}{2}$.

Lemma \ref{l.adequatebound}, which  says that  $\deg \langle \sk_\sigma(a) \rangle \leq \deg \langle \overline{\sk_\sigma(a)} \rangle$, gives us
$$\deg \left( I(a, -1, n) A^{\sgn(\sigma)}  \langle \sk_{\sigma}(a) \rangle \right) \leq  d(a, -1, n) + \sgn(\sigma) + \deg \langle \overline{\sk_{\sigma}(a)} \rangle,$$
where, as defined in Definition \ref{d.barS}, $\overline{\sk_{\sigma}(a)}$ is the skein element obtained from $\sk_{\sigma}(a)$ by replacing every Jones-Wenzl projector by the identity element. 

For fixed $n$,  clearly $ d(a, -1, n)$ increases monotonically in $a$. Similarly for fixed $n$ and $\sigma$, the function $\deg \langle \overline{\sk_{\sigma}(a)} \rangle$ increases monotonically in $a$, since $ \overline{\sk_{\sigma}(a)}$ is just a collection of disjoint circles whose number is determined by $\overline{\mathcal{J}_\sigma(a)}$, and  the number of circles increases with $a$. See Figure \ref{f.countcircles}.
\begin{figure}[H]
\def \svgwidth{.15\columnwidth}
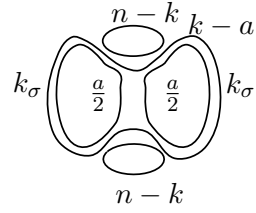 
\caption{\label{f.countcircles}  The crossingless skein $\overline{J_\sigma(a)}$.  }
\end{figure} 
Thus, since we work with $\frac{a}{2} \leq k_{\sigma}$, we can isolate the terms that contribute to the highest degree of the sum \eqref{e.dn}, to rewrite

\begin{equation}
\langle \Dj^n \rangle =
 \sum_{\substack{\sigma} }\sum_{a \leq 2k_{\sigma}} I(a =2k_{\sigma} , -1, n) A^{\sgn(\sigma)}  \langle \overline{\mathcal{S}_{\sigma}(a=2k_{\sigma})} \rangle 
+ \text{ lower degree order terms.}
\label{high}
\end{equation}

 We will distinguish two cases:
 \vskip 0.05in
 
 \noindent\underline{ {\bf Case 1.}}
Suppose we have  $k_\sigma = 0$, for all $\sigma$ of Equation \eqref{high}. Then the only nonzero term in the highest degree terms of the sum of \eqref{high}  is when $a = 2k_\sigma = 0$. In this case, applying Corollary \ref{c.ub} we have
\begin{align*}
 \deg\left( I(0, -1, n) A^{\sgn(\sigma)}  \langle \overline{\mathcal{S}_{\sigma}(0)}\rangle \right) &\leq   \deg\left( I(0, -1, n) A^{\sgn(\sigma_A)}  \langle \overline{\mathcal{S}_{\sigma_A}(0)}\rangle \right).
 \intertext{Now we compute}
 \deg\left( I(0, -1, n) A^{\sgn(\sigma_A)}  \langle \overline{\mathcal{S}_{\sigma_A}(0)}\rangle \right)&= \deg(I(0, -1, n)) + \sgn(\sigma_A) + \deg \langle \overline{\sk_{\sigma_A}(0)} \rangle= \\ 
  &= d(0, -1, n) + (c(D)-1)n^2 + 2v_A(D) n + 2n =\\ 
  &= -4n -n^2 + (c(D)-1)n^2 + 2v_A(D) n + 2n. \\ 
 \intertext{With $H_n(D) = d(2n, -1, n) +  \sgn(\sigma_A) + \deg\langle \overline{\sk_A(2n)}\rangle$ we have  } 
 H_n(D) - \deg \langle \Dj^n \rangle &\geq 2n^2 + 2n. 
\end{align*}
The statement of the theorem follows from setting $\rho = 2$. 

 \vskip 0.05in
 
 \noindent\underline{ {\bf Case 2.}}
Suppose now that there exists some $\sigma$ for which $k_\sigma \not=0$. 
Discarding the lower degree terms and replacing the parameter $a$ by $2k_{\sigma}$ in \eqref{high},
we have
\begin{equation}
\deg \langle \Dj^n \rangle  \leq   \max_{\substack{\sigma } }\{ d(2k_\sigma, -1, n)+  \sgn(\sigma) + \deg \langle \overline{\sk_\sigma(2k_{\sigma})} \rangle\}.
\label{e.max}
\end{equation}

Let $v\in \mathbb{G}_A$ denote the vertex to the one edge loop $e$
corresponding to the crossing $x$ we fixed at the start of the proof.
This vertex corresponds to a state circle in $\sigma_A(\overline{\Dj^n})$ where, using the notation
of   Definition \ref{d.barS}, we have $\overline{\Dj^n} = D^n$.  
Since any $\sk_A(a)$ is obtained by applying the all-$A$ Kauffman state outside the fused edge, the number of through strands $k_{\sigma_A}$ is zero.
Thus we can only have a non-zero number of through strands in $\sk_\sigma(a)$, where $\sigma \not= \sigma_A$.
In other words, through strands come from state arcs of  $\sigma(\Dj^n)$, created from the choice of the $B$-resolution on a crossing by  $\sigma$. By Remark \ref{r.throughe}, to every though strand we can associate a finite closed walk in $\mathbb{G}_A$ starting and ending at $v$: The sequence of edges in the walk are the edges $\{e_y\}_y$, correspond to crossings $\{y\}$ that the  strand passes through, in the sense of Definition \ref{d.pass}.  Moreover, an edge in a walk coming from a through strand can repeat at most $2n$ times, since the maximal number of through strands possible for any state $\sigma$ on $y^n$ is $2n$.

To facilitate exposition we will use $\Sigma_{\max}$ to denote the set of states that contribute to the right hand side of (\ref{e.max}).
For $\sigma \in \Sigma_{\max}$ and a crossing $y$ of $D$, let $k_y$ 
denote half the number of through strands  of the skein  $\sigma(y^n)$  obtained by restricting
$\sigma|y^n$.  Also let $ \sigma_A|y^n$ denote the restriction of the all-$A$ state on $y^n$.

By Lemma \ref{l.xoriented}, since $\sigma(y^n)$ has $2k_y$ through strands, a sequence of states from $\sigma_A|y^n$ to $\sigma|y^n$
contains a subsequence of length $k^2_y$. This means that the number of crossings on which $\sigma$ chooses the $B$-resolution in $y^n$ is at least $k_y^2$, and therefore, 
\begin{align} \label{e.xlowerbound} 
\sgn(\sigma_A|y^n) - \sgn(\sigma|y^n) \geq 2k_y^2.
\end{align} 

Recall  $H_n(D): = \sgn(\sigma_A) + d(2n, -1, n) + \deg \langle \overline{ \sk_A(2n)}\rangle$. For $\sigma \in \Sigma_{\max}$, by   \eqref{e.max}, we have
 
 \begin{align}
  \label{e.10}
&H_n(D) - \deg \langle \Dj^n \rangle \geq  d(2n, -1, n)-d(2k_{\sigma}, -1, n)+ \\
 &+ \big( \sgn(\sigma_A)- \sgn(\sigma)\big)+\big(\deg \langle \overline{ \sk_A(2n)}\rangle- \deg \langle \overline{\sk_\sigma(2k)}\rangle\big).\notag 
 \end{align} 

Consider the set $W=\{w_1, \cdots, w_s\}$ of closed walks on $\mathbb{G}_A$  from $v$ to $v$,  such that the number of times that each edge of $\mathbb{G}_A$  appears in a walk in $W$
is at most $2n$, and the walk consisting only of $e$ is not in $W$. Let $s$ denote the cardinality of  $W$.

Given  $\sigma \in \Sigma_{\max}$, for $i = 1\leq \cdots \leq s$,  let $k_i$ denote half the number of through strands of  $\mathcal{T}_\sigma$ corresponding to  the walk $w_i$. Here $k_i$ is not necessarily an integer. The set $\{2k_1, \ldots, 2k_s \}$ gives a nonnegative integer partition of $2k_\sigma$ into $s$ parts.  For $i = 1 \leq \cdots \leq s$, pick an edge on the  walk $w_i$. This gives a  set of crossings $\{y_1, \ldots, y_s \}$ on the knot diagram $D$. Since we assume $k_\sigma \not=0$, we have $s > 0$. Now remove any repeated edges $y_j$ and renumber, to get a set of $s'$ distinct edges $\{y_1, \ldots, y_{s'} \}$.   

For $j=1\leq \cdots\leq s'$, set $k'_j=k_{y_j}$ and recall that a sequence of states from $\sigma_A$ to $\sigma$ must contain a subsequence of length at least $(k'_j)^2$.
Note that by our choice of  the set $\{y_1, \ldots, y_{s'}\}$ the $s'$ subsequences we get this way are distinct. Moreover, every through strand of $\mathcal{T}_\sigma$ is contained in $y_j$ for some $j$, so $k' = \sum k_j' \geq k_\sigma$.

Now using \eqref{e.xlowerbound}  we get
 $$\sgn(\sigma_A) - \sgn(\sigma) \geq 2\sum_{j=1}^{s'} (k'_j)^2.$$ Furthermore,  since for each $j=1\leq \cdots\leq s'$ the corresponding subsequence can create at most $2k'_j$ new state circles,
we obtain $\deg \langle \overline{\sk_{A}(2n)}\rangle - \deg \langle \overline{\sk_\sigma(2k)} \rangle \geq -2\sum_{j=1}^{s'} k'_j$. 
Finally, by Lemma  \ref{l.Twistdegree},  we have $d(2k_{\sigma}, -1, n) = -4n + 4k_{\sigma} - n^2 + 2k_{\sigma}^2$.
With these observations  at hand,  (\ref{e.10}) leads to
 \begin{align} \label{e.finequality}
  &H_n(D) - (d(2k_{\sigma}, -1, n) + \sgn(\sigma) + \deg \langle \overline{\sk_\sigma(2k_{\sigma})}\rangle)\geq \\
 &\geq  2(n^2-k_{\sigma}^2) +4(n-k_{\sigma}) + \sum_{j=1}^{s'} 2((k'_j)^2 - k'_j). \notag
 \end{align}

Now $\{k'_1, \ldots, k'_{s'}\}$ is a nonnegative integer partition of  $k': = \sum_{j=1}^{s'} k'_j$.
Applying Lemma \ref{elementarym} to $k'$ and $s'$ we get a minimal partition of $k'$ into $s'$ part as in \eqref{e.minpartition}.
Then, applying Lemma \ref{l.partition}
to compare the resulting minimal partition and  to $\{k'_1 \ldots k'_{s'}\}$, 
we have
$$
b ( k'/s' + 1 - b/s')^2 + (s-b) (k'/s'- b/s')^2 \geq (k')^2/s' = \sum_{j=1}^{s'}  (k'/s')^2.  $$
Now from  \eqref{e.finequality} we obtain
\begin{align}
&H_n(D) - (d(2k_{\sigma}, -1, n) + \sgn(\sigma) + \deg \langle  \overline{\sk_\sigma(2k_{\sigma})}\rangle)\geq \notag   \\ 
   &\geq  2(n^2-k_{\sigma}^2) +4(n-k_{\sigma}) + 2\sum_{j=1}^{s'} \left(\frac{k'}{s'}\right)^2-  2\sum_{j=1}^{s'} k'_j =\notag \\ 
      &=  2(n^2-k_{\sigma}^2) +4(n-k_{\sigma}) + 2\frac{(k')^2}{s'}-  2\sum_{j=1}^{s'} k'_j.\notag 
 \end{align}
 
 Recall $k'\geq k_\sigma$ by construction of the set $\{y_1, \ldots, y_{s'}\}$. Also, $s\geq s'$ and $ns'\geq k'$ since $k'_j \leq n$. Hence we get
\begin{align}   &H_n(D) - (d(2k_{\sigma}, -1, n) + \sgn(\sigma) + \deg \langle \overline{\sk_\sigma(2k_{\sigma})}\rangle) \geq \notag \\ 
  &\geq 2(n^2-k_{\sigma}^2) +4(n-k_{\sigma}) + \frac{2k_{\sigma}^2}{s}- 2ns \notag=\\
  &=\big(\frac{2-2s}{s}\big)k_{\sigma}^2-4k_{\sigma}+(2n^2+4n-2ns). \notag
  \end{align}
 
Denote the quantity on the right hand side of the  last inequality   
 by $g(k)$, where $k:=k_{\sigma}$.

By assumption $s\geq 1 $. Since $s$   is the cardinality of a set of walks on $ \mathbb{G}_A$ it is independent from $\sigma$. Thus for  fixed $n$, $g(k)$ is  a function of $k$, where $0\leq k\leq n$.

If $s=1$, then $g(k)$ is a linear function in $k$ with negative derivative -4, hence on $[0, n]$ it achieves its absolute minimum $2n^2 + 2n$ on $n$.

Otherwise assume $s\geq 2$. Since the critical point of $g(k)$ is $\frac{s}{1-s}<0$, the absolute minimum in $[0, n]$  is achieved at $k=0$ or  $k=n$.  Hence
$g(k)\geq \min\{ 2n^2-2ns, \ \frac{2n^2}{s} -2ns\}.$
Thus  setting $\rho = \frac{2}{s}>0$ we have that $g(k) \geq \rho n^2 + O(n)$ and the conclusion follows.
\end{proof}

%%%%%%%%%%%%%

\section{Applications to crossing numbers}
Determining the crossing number of  an arbitrary  knot $K$ is a hard task as there are no general methods for it other than a brute-force search  that would attempt to classify knots  that admit diagrams with crossings less than or equal these of a  diagram for $K$.
Such methods have been used successfully to compile tables of knots with low crossing numbers \cite{Hoste} but become hopeless for arbitrary knots.
Although there has been some progress in understanding  the behavior of the crossing number under  the operations of taking  connected sums or forming satellites of knots \cite{Lackenby},
fundamental questions in this direction still remain out of reach \cite[Problems 1.67, 1.68]{Kirby}.

On the other hand  the crossing numbers for broad families of knots that admit particular types of diagrams are well understood.
In particular, it is known that  adequate  diagrams realize the crossing number  of the knots they represent, and that the crossing number of adequate knots is additive under connected sums \cite{Murasugi, spanning, Kauffmanstates}. In addition, it is known  that the ``standard" diagrams of Montesinos knots
and torus knots minimize their crossing number \cite{LickorishThistlethwaite, Lickorishbook, Adams}.
As a Corollary of Theorem \ref{main} we obtain the following  criterion that allows to determine the crossing number of non-adequate knots that admit diagrams  with the number of crossings ``close enough"  to their Jones diameter.
\begin{cor}\label{criterion} Suppose $K$ is a non-adequate knot admitting a diagram $D=D(K)$ such that  $jd_K=2(c(D)-1)$. Then we have $c(D)=c(K)$.
\end{cor}
\begin{proof} Since $K$ is non-adequate, Theorem \ref{main} gives $$c(D)\geq c(K)> \frac{ jd_K}{2}=c(D)-1,$$ and the result follows.
\end{proof}

Next we will discuss  lower bounds for the  crossing number of Whitehead doubles of adequate and torus knots.  Using Corollary \ref{criterion} we will 
determine the crossing numbers of infinite families of Whitehead doubles.

\subsection{Doubles of adequate knots} 
Let $V$ be a standard solid torus in $S^3$, with preferred meridian-longitude  pair $(\mu_V, \lambda_V)$ and with $U_{\pm}$ a copy of a $\pm$-clasped unknot as shown in Figure \ref{f.WD}.
Given $K\subset S^3$ with a torus neighborhood $V_K$ and preferred meridian-longitude  pair $(\mu_K, \lambda_K)$,
take an embedding $f: V \longrightarrow S^3$ with $f(V)=V_K$, $f(\mu_V)=\mu_K$ and $f(\lambda_V)=\lambda_K$.
Then $W_{\pm}(K):=f(U_{\pm})$ is  the untwisted (positive/negative) Whitehead double of $K$.

We recall that if $D=D(K)$ is an adequate diagram, and with the notation of Definition \ref{d.stategraph}, the quantities $c(D)$, $c_{\pm}(D)$
as well as the Turaev genus $g_T(D)$ are minimal over all diagrams representing $K$
 \cite{Lickorishbook,  Kauffmanstates, Murasugi, alternating, Abe}.
We denote them by
$c(K)$, $c_{\pm}(K)$, and $g_T(K)$, respectively.
Furthermore,  the  writhe number of $D$ is known to be an invariant of $K$ and is denoted by
 $\mathrm{wr}(K)$.
 
In this section we prove the following result which, as we will explain later on, implies in particular Theorem \ref{doubleintro} stated in the Introduction. 

\begin{figure}[H]
\begin{center}
\def \svgwidth{.45\columnwidth}
%% Creator: Inkscape 1.0beta1 (32d4812, 2019-09-19), www.inkscape.org
%% Accompanies image file 'WD.pdf' (pdf, eps, ps)
\begingroup%
  \makeatletter%
  \providecommand\color[2][]{%
    \errmessage{(Inkscape) Color is used for the text in Inkscape, but the package 'color.sty' is not loaded}%
    \renewcommand\color[2][]{}%
  }%
  \providecommand\transparent[1]{%
    \errmessage{(Inkscape) Transparency is used (non-zero) for the text in Inkscape, but the package 'transparent.sty' is not loaded}%
    \renewcommand\transparent[1]{}%
  }%
  \providecommand\rotatebox[2]{#2}%
  \newcommand*\fsize{\dimexpr\f@size pt\relax}%
  \newcommand*\lineheight[1]{\fontsize{\fsize}{#1\fsize}\selectfont}%
  \ifx\svgwidth\undefined%
    \setlength{\unitlength}{656.99375423bp}%
    \ifx\svgscale\undefined%
      \relax%
    \else%
      \setlength{\unitlength}{\unitlength * \real{\svgscale}}%
    \fi%
  \else%
    \setlength{\unitlength}{\svgwidth}%
  \fi%
  \global\let\svgwidth\undefined%
  \global\let\svgscale\undefined%
  \makeatother%
  \begin{picture}(1,0.54248883)%
    \lineheight{1}%
    \setlength\tabcolsep{0pt}%
    \put(0,0){\includegraphics[width=\unitlength,page=1]{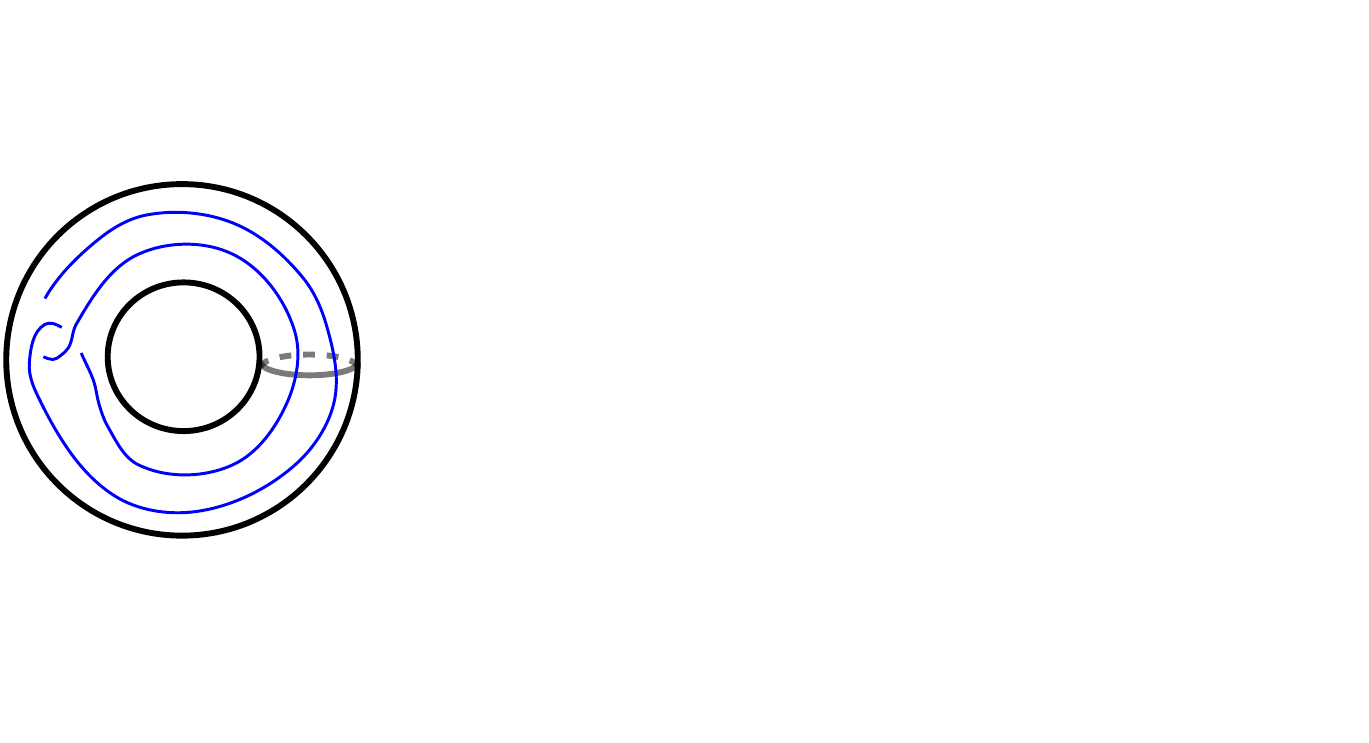}}%
    \put(0.01571232,0.12095826){\makebox(0,0)[lt]{\lineheight{1.25}\smash{\begin{tabular}[t]{l}$V$\end{tabular}}}}%
    \put(0.2054247,0.12107347){\makebox(0,0)[lt]{\lineheight{1.25}\smash{\begin{tabular}[t]{l}$U_{+}$\end{tabular}}}}%
    \put(0.27941346,0.24843197){\makebox(0,0)[lt]{\lineheight{1.25}\smash{\begin{tabular}[t]{l}$\mu_V$\end{tabular}}}}%
    \put(0,0){\includegraphics[width=\unitlength,page=2]{WD.pdf}}%
    \put(0.34219942,0.30551014){\makebox(0,0)[lt]{\lineheight{1.25}\smash{\begin{tabular}[t]{l}$f$\end{tabular}}}}%
    \put(0,0){\includegraphics[width=\unitlength,page=3]{WD.pdf}}%
  \end{picture}%
\endgroup%
 
\caption{\label{f.WD}The positive Whitehead double of the figure-8 knot. By Theorem \ref{doubleadequate} the diagram shown is a minimal crossing diagram. }
\end{center}  
\end{figure}

\begin{thm}\label{doubleadequate} Suppose that $K$ is an adequate knot with crossing number $c(K)$ and writhe $\mathrm{wr}(K)$.
Suppose moreover that $c_+(K), c_-(K)\neq 0$ and let $W_{-}(K)$ (resp. $W_{+}(K)$) denote the negative (resp. positive) untwisted Whitehead double of $K$. 
Then, the crossing number $c(W_{\pm}(K))$ satisfies the following inequalities.
$$4c(K)+1\leq c(W_{\pm}(K))\leq   4c(K)+2+2|\mathrm{wr}(K)|.$$

Furthermore, if $\mathrm{wr}(K)=0$ 
 we have $c(W_{\pm}(K))=4 c(K)+2$ and the diagram $W_{\pm}(D)$, formed by doubling  an adequate diagram $D=D(K)$ using the blackboard framing of $D$, is a minimum
crossing diagram for $W_{\pm}(K)$.
\end{thm}

Note that the lower bound of Theorem \ref{doubleadequate} is  sharper than the general prediction stated in \cite[Problem 1.68]{Kirby} and the one
announced in the unpublished preprint \cite{Pascual}. For the proof of Theorem \ref{doubleintro} it is crucial that we have the  sharper lower bound of Theorem \ref{doubleadequate}  and Theorem \ref{main}.

 For the  proof of Theorem \ref{doubleadequate}  we will use the following result of Baker, Motegi and Takata 
 which is a special case of \cite[Proposition 2.4]{BMT}\footnote{Their result more generally assumes that $d_+[J_K(n)]$ is a quadratic quasi-polynomial $a(n)n^2 + b(n)n + c(n)$ for all $n\geq 1$ of period $\leq 2$, with $a_1:= a(1)$, $b_1:= b(1)$, and $c_1:= c(1)$. In our notation, these assumptions are satisfied for the adequate knots we work with in this paper for whom $d_+[J_K(n)]$ is a quadratic polynomial in $n$ for all $n\geq 1$. }.   
 
 \begin{prop}\label{BMTprop} Suppose that $K$ is a knot such that
 $ d_+[J_K(n)]=a_2 n^2+a_1n +a_0$ is a quadratic polynomial for all $n>0$.  Suppose that $a_1\leq 0$ and that if $a_1=0$ then $a_2\neq 0$.
 Then we have the following.
 \begin{enumerate}[(a)]
 \item \label{BMTpropa} If $a_2>0$ then, for $n$ sufficiently large,
 $$ d_+[J_{W_-(K)}(n)]=4a_2 n^2+ (-4a_2+2a_1-\frac{1}{2}) n+(a_2-a_1+a_0+\frac{1}{2}).$$
 
 \item \label{BMTpropb}  If $a_2>\frac{1}{8}$ then, for $n$ sufficiently large,
 $$ d_+[J_{W_+(K)}(n)]=(4a_2+\frac{1}{2}) n^2+ (-4a_2+2a_1) n+(a_2-a_1+a_0-\frac{1}{2}).$$
 \end{enumerate}
  \end{prop}

  \begin{proof} Following the conventions and notation of 
\cite[Proposition 2.4]{BMT} we take $\tau=0$, $w=1$, for $d_+[J_{W_-(K)}(n)]$ and
$\tau=0$, $w=-1$, for $d_+[J_{W_+(K)}(n)]$.
  \end{proof}

 A key part in the proof of Theorem \ref{doubleadequate}  is to show that if $\mathrm{wr}(K)=0$,
 then
 the Whitehead double $W_{\pm}(K)$ is non-adequate. This task is accomplished in the next two lemmas, using
 Proposition \ref{BMTprop} and properties of adequate knots and of their colored Jones polynomials. Then, Theorem \ref{doubleadequate}
 will follow easily from Corollary \ref{criterion}. Note that our first lemma doesn't require the hypotheses $\mathrm{wr}(K)=0$.

 \begin{lem}\label{ifadequate} Let $K$ be an adequate knot with crossing number $c(K)$ and writhe $\mathrm{wr}(K)$. Suppose moreover that $c_+(K), c_-(K)\neq 0$.
  If $W_{\pm}(K)$ is adequate, then
 $$c(W_{\pm}(K))=4c(K)+1\ \  \text{and}  \ \ g_T(W_\pm(K))=c(K)+2 g_T(K)-1,$$ where $g_T(W_\pm(K))$ denotes the Turaev genus of $W_\pm(K)$.
 \end{lem}

 \begin{proof} Since $K$ is adequate, by Lemma \ref{known}, 
 \begin{equation}
 d_+[J_K(n)] - d_-[J_K(n)] = \frac{c(K)}{2} n^2 + (1-g_T(K) - \frac{c(K)}{2})n + g_T(K) - 1, 
 \label{eq:ad}
 \end{equation}
 for every $n\geq 0$. 
 Furthermore, $ d_+[J_K(n)]$ 
satisfies 
 the hypothesis of Proposition \ref{BMTprop} with $4a_2=2c_+(K)>0$ and
  $  d_-[J_K(n)]=- d_+[J_{K^{*}}(n)]$ satisfies  that hypothesis with $4a_2=2c_+(K^{*})=2c_-(K)$. 
The requirement that $a_1\leq 0$ is satisfied since for adequate knots the linear terms of the degree of $J_K^{*}(n)$ are multiples of Euler characteristics of spanning surfaces of $K$.
See \cite[Lemmas 3.6, 3.7]{KTran}. Finally we have $a_2\neq 0$ since the statement of Theorem \ref{doubleadequate} assumes that $c_+(K), c_-(K)\neq 0$.
 Now  Proposition \ref{BMTprop}  implies that for sufficiently large $n$ we have that $d_+[J_{W_{\pm}(K)}(n)]-d_-[J_{W_{\pm}(K )}(n)]$ is actually a quadratic polynomial. 
 That is, there is some $n_0$, depending on $W_{\pm}(K)$, so that for all $n>n_0$, we have 
 $$
 d_+[J_{W_{\pm}(K)}(n)]-d_-[J_{W_{\pm}(K)}(n)]=d_2 n^2+d_1 n+d_0,
$$
 with $d_i \in {\mathbb Q}$.
 Using Proposition \ref{BMTprop}, the fact that $d_+[J_{W_+(K^{*} )}(n)]=-d_-[J_{W_-(K)}(n)] $, and Equation (\ref{eq:ad}), we will compute the constant  $d_1+d_2$
 in terms of the coefficients of $d_+[J_K(n)]-d_-[J_K(n)]$.
 
 To that end, write
$d_+[J_K(n)] =a_2 n^2 +a_1 n +a_0$ and $-  d_-[J_K(n)] =a^{*}_2 n^2 +a^{*}_1 n +a^{*}_0$ .
By Equation \eqref{eq:ad}  we have $a_2 +a^{*}_2=\frac{c(K)}{2}$  and $ a^{*}_1+a_1=1-g_T(K) - \frac{c(K)}{2}$.

We have $ d_2=4a_2+4a^{*}_2+\frac{1}{2}=2c_+(K)+2c_-(K)+\frac{1}{2}=2c(K)+\frac{1}{2}$, and
$$
d_1+d_2= 2a_1+2a^{*}_1= 2(1-g_T(K) - \frac{c(K)}{2})=2-2g_T(K)-c(K). 
$$

Now if $W_{\pm}(K)$ is adequate, then, again by Lemma \ref{known}, we also have
$
d_2=\frac{c(W_{\pm}(K))}{2}$ and
$d_1+d_2=1-g_T(W_{\pm}(K))$. Now comparing the 
right hand sides of the two expressions we have for $d_2$ and for $d_1+d_2$ we get the desired results.   \end{proof}
  
  Next we show that, under the additional hypothesis that  $\mathrm{wr}(K)=0$  the knots $W_{\pm}(K)$ are  non-adequate.
  
 \begin{lem}\label{notadequate} Let  $K$ be a nontrivial  adequate knot with
 $\mathrm{wr}(K)=0$. Then, the untwisted Whitehead doubles $W_{\pm}(K)$ are  non-adequate.  \end{lem}
 
 \begin{proof} We will work with the negative  Whitehead doubles $W_{-}(K)$ first.
 
Recall that if   $K$ has an adequate diagram $D=D(K)$ with $c(D)=c_+(D)+c_-(D)$ crossings, and the all-$A$  (rep. all-$B$) resolution
 has $v_A=v_A(D)$ (resp.  $v_B=v_B(D)$) state circles, then 
 
  \begin{equation}
 4\, d_-[J_{K}(n)] =  -2c_- (D) n^2 + 2(c(D) -v_A(D)) n  +2 v_A(D) -2 c_+(D),
\label{top}
\end{equation}
 \begin{equation}
4 \, d_+[J_{K}(n)] = 2c_+ (D)n^2 + 2(v_B(D) - c(D)) n +2 c_-(D)-  2v_B(D).
\label{bottom}
\end{equation}
Equation (\ref{top}) holds for $A$-adequate diagrams $D=D(K)$.
Thus in particular the quantities $c_- (D), v_A(D)$ are invariants of $K$ (independent of the particular $A$-adequate diagram).
Similarly, Equation (\ref{bottom}) holds for $B$-adequate diagrams $D=D(K)$ and hence
$c_+ (D), v_B(D)$ are invariants of $K$. Recall also that $c(D) = c(K)$ since $D$ is adequate. 

Now we start with a knot $K$ that has an adequate diagram $D$ with $\mathrm{wr}(D)=\mathrm{wr}(K)=0$. Hence we have $c_+ (D)=c_- (D)$. Since $D$ is $B$-adequate, the double   $W_{-}(D)$ is a $B$-adequate diagram of  $W_{-}(K)$ with $v_B(W_{-}(D))=2v_B(D)+1$ and $c_+(W_{-}(D))=4c_+(D)$.  These  statements are proved, for instance, in \cite[Proposition 7.1]{BMT}.
Furthermore, since as said above these quantities are invariants of $W_{-}(K)$, they remain the same for all $B$-adequate diagrams of
$W_{-}(K)$.

Now assume, for a contradiction, that $W_{-}(D)$ is adequate: Then, it has a diagram ${\bar D}$ that is both $A$ and $B$-adequate.  
By above observation we must have
$v_B({\bar D})=v_B(W_{-}(D))=2v_B(D)+1$ and $c_+({\bar D})=c_+(W_{-}(D))=4c_+(D)$.

By Lemma \ref{ifadequate}, 
$c({\bar D})=c(W_{-}(K))=4c(K)+1$ and, since   $g_T({\bar D})=g_T(W_-(K))$ \cite{Abe}, we also obtain \begin{equation}
g_T({\bar D})=g_T(W_-(K))=c(K)+2 g_T(K)-1.
\label{TG}
\end{equation}
Now we compare the two expressions of $g_T(W_-(K))$ in Equation (\ref{TG}) in order to get a relation between  $v_A({\bar D})$ and $v_A(D)$.

On  one hand, $2g_T({\bar D})=2-v_B({\bar D})-v_A({\bar D})+c({\bar D})=2-2v_B(D)-1-v_A({\bar D})+4c(D)+1=2-2v_B(D)-v_A({\bar D})+4c(D)$.

On the other hand, $2g_T(W_-(K))=2(c(K)+2 g_T(K)-1)=2c(D)+2(2-v_A(D)-v_B(D)+c(D))-2=2-2v_B(D)-2v_A(D)+4c(D)$.

Comparing the right-hand sides of the last two equations we find
$v_A({\bar D})=2v_A(D)$.

Write
$$
- 4\, d_-[J_{W_-(K)}(n)] =4\, d_+[J_{W_+(K^{*})}(n)]=xn^2+yn+z,
$$
for some  $x,y,z\in {\mathbb Q}$.

For sufficiently large $n$, we have two different expressions  for $x,y,z$. 

On the one hand, because ${\bar D}$ is $A$-adequate, we can use  Equation (\ref{top}) to determine $x,y,z$ through $- 4\, d_-[J_{W_-(K)}(n)]$.

On the other hand, using $4\, d_+[J_{W_+(K^{*})}(n)]$,   $x,y,z$ can be determined using Proposition \ref{BMTprop}\eqref{BMTpropb} with $a_2$ and $a_1$ coming from Equation (\ref{top}) applied to the diagram $D^{*}$, that is the mirror image of $D$.

We will use these two ways to find the quantity $y$.
Applying Equation (\ref{top})  to ${\bar D}$ we obtain
\begin{equation}
y=2(c({\bar D}) -v_A({\bar D})) =2(4c(D)+1))-2 (2v_A(D))=
8 c(D)-4 v_A(D) +2.
\label{eq:33}
\end{equation}

On the other hand,  using  Proposition \ref{BMTprop}\eqref{BMTpropb} with $a_2$ and $a_1$ coming from Equation (\ref{top}) applied to the diagram $D^{*},$ we have:
$4a_2=-2c_-(D^{*})=-c_+(D)-c_-(D)=-c(D)$. Also we have
$2a_1=c(D^{*})-v_A(D^{*})=c(D)-v_B(D)$, since we have $v_A(D^{*})=v_B(D)$.

We obtain
\begin{equation}
y=4(-4a_2+2a_1)=4c(D)+4(c(D)-v_A(D^*))=8c(D)-4v_B(D).
\label{eq:34}
\end{equation}

Since $v_A(D), v_B(D)$ are positive integers we have $-2v_A(D)+1\neq -2v_B(D)$. It follows that the two expressions derived for $y$  from Equations \eqref{eq:33} and  \eqref{eq:34} do not agree and we arrived at a  contradiction. We conclude that $W_{-}(K)$ is  non-adequate.

To deduce the result for $W_{+}(K)$, let $K^{*}$ denote the mirror image of $K$. Note that $W_{+}(K)$ is the mirror image of $W_{-}(K^{*})$.
 If $W_{+}(K)$ were adequate, then the mirror image $W_{-}(K^{*})$ would also be adequate. But $K^{*}$ is a non-trivial adequate knot with 
 $\mathrm{wr}(K^{*})=\mathrm{wr}(K)=0$, and our argument above shows that $W_{-}(K^{*})$ is non-adequate.
  \end{proof}

  We now finish the proof of   Theorem \ref{doubleadequate}.
  
   \begin{proof} 
    Let $D$ be an adequate diagram of $K$ with writhe $\writhe(D)$.  If needed, first adjust $D$ by Reidemeister I moves so that it has zero writhe number.
    Then, let $W_{-}(D)$ (resp.  $W_{+}(D)$) be the diagram of
    $W_{-}(K)$ (resp.  $W_{+}(K)$)  obtained by taking a parallel copy of $D$ and connecting the two copies by  a negative (resp. positive) clasp.
  Clearly $W_{\pm}(D)$ has $4c(K)+2|\mathrm{wr}(K)|+2$ crossings.  Thus the upper inequality follows. 
      
    As discussed in the proof of Lemma \ref{ifadequate},  we have $$\frac{jd_{W_{\pm}(K)}}{2}=2(2c(D)+\frac{1}{2})= (4c(K)+1)\geq c(W_{\pm}(D)) -1-2\mathrm{wr}(K),$$ 
   and hence, if $\mathrm{wr}(K)=0$, we get $\frac{jd_{W_{\pm}(K)}}{2}=c(W_{\pm}(D)) -1$.
    On the other hand, by Lemma \ref{notadequate}, if     $\mathrm{wr}(K)=0$ then  $W_{\pm}(K)$ is non-adequate. 
    Thus, if $\mathrm{wr}(K)=0$,  Corollary \ref{criterion} applies to give $c(W_{\pm}(K))=c(W_{\pm}(D))=4c(K)+2$.
    \end{proof}
   
   \subsection{ Doubles of amphicheiral knots}
Note that  amphicheiral (a.k.a. equivalent to their mirror image)  adequate knots must have $\mathrm{wr}(K)=0$. This means that amphicheiral  adequate knots admit adequate diagrams of zero writhe number.  By Theorem \ref{doubleadequate} we have the following.
\begin{cor} \label{amphi} Suppose that $K$ is an amphicheiral  adequate knot with crossing number $c(K)$. Then $c(W_-(K))=4 c(K)+2.$
\end{cor}

%It is known that for any $c>0$, there are alternating, amphicheiral  knots of crossing number $c$ \cite{TaitIV}.
The figure-8 knot is the fist example on the knot table to which Corollary \ref{amphi} applies.  For $m>0$,
letting $K_m$ denote the connected sum of $m$-copies of the  figure-8 knot we have
$c(W_{\pm}(K_m))=16m+2$. This should be compared with the discussion in \cite{blog}.

For completeness, in Table 1 we give the  list of all the prime knots up to 12 crossings to which Corollary \ref{amphi} applies.
The information is taken from Knotinfo  \cite{Knotinfo}, where the terminology used for knots that are equivalent to their mirror images is \emph{fully amphicheiral}.
\begin{table}[H]
\begin{center}
\def \svgwidth{.4\columnwidth}
\begin{tabular}{|c|c|c|c|c|c|c|}
\hline
$4_1$ &  $8_{18}$  &     $10_{43}$ &    $12a_{435}$ &  $12a_{506}$  &  $12a_{1105}$    &   $12a_{1275}$   \\
\hline
 $6_3$ &$10_{17}$&  $10_{45}$ &      $12a_{471}$ &   $12a_{510}$  &   $12a_{1127}$      & $12a_{1281}$ \\
\hline
 $8_3$ & $10_{33}$ &     $10_{99}$ &    $12a_{477}$  &  $12a_{1019}$  & $12a_{1202}$       &  $12a_{1287}$ \\
\hline
 $8_9$ & $10_{37}$&     $10_{123}$ &  $12a_{499}$ &     $12a_{1039}$  &  $12a_{1273}$   & $12a_{1288}$    \\
\hline
  \end{tabular}
\caption{Prime, adequate, amphicheiral knots up to 12 crossings}
\end{center}

 \end{table}

A way to produce amphicheiral  knots is to take connected sums of knots with their mirror images as in Corollary \ref{mirror} which we prove  next.

\begin{named}{Corollary \ref{mirror}} For a knot $K$ let $K^{*}$ denote the mirror image of $K$.  For every  $m>0$, let $K_m:=\#^{m}(K\# K^{*})$ denote the connected sum of $m$-copies of $K\# K^{*}$. Suppose that $K$ is adequate with crossing number $c(K)$. Then, the untwisted Whitehead doubles $W_{\pm}(K_m)$ are non-adequate, and we have
$c(W_{\pm}(K))=8 m c(K)+2.$
\end{named}

\begin{proof}Given an adequate diagram $D=D(K)$, the mirror image of $D$, denoted by $D^{*}$ is an adequate diagram for $K^{*}$.
The connected sum of adequate diagrams $D\# D^{*}$ is an adequate diagram of $K\#K^{*}$  with $\mathrm{wr}(K\#K^{*})=0$.  Note, that the choice
of orientations of $K$ and $K^{*}$ may affect the (oriented) knot type $K\#K^{*}$. Nevertheless, with any choice of orientations, the knot represented by the connected sum is adequate.

Similarly an adequate diagram of writhe zero for $K_m$ is obtained by taking the connected sum of $n$-copies of  $D\# D^{*}$.
Since the crossing number is known to be additive under connected sum of adequate knots, we have $c(K_m)= 2 m c(K)$
and the result follows from Corollary \ref{amphi}.
In fact, we get that $W_{\pm}(D_m)$ is a minimum crossing diagram for $K_m$.\end{proof}

\begin{remark}
Out of the 2977 prime knots with up to 12 crossings, 1851 are listed as adequate on Knotinfo  \cite{Knotinfo}. Hence, Corollary \ref{mirror} applies to them.
\end{remark}

 \subsection{Doubles of torus knots} For  co-prime integers $p, q$ with $1<p, q$, let $T_{p,q}$ denote the $(p,q)$ torus knot.
  It is known that $c(T_{p,q})={\rm min}\{ p(q-1), \ q(p-1)\}$. On the other hand, as it can be found  for example in \cite{Garslopes},  we have that $jd_{K}=pq<2c(T_{p,q})$, for $q>2$.
  Thus $T_{p,q}$ is not adequate for $q>2$.  
  
  \begin{prop} We have $c(W_{\pm}(T_{p,q}))> 2c(T_{p,q}).$
  \end{prop}  
  \begin{proof} The Jones diameter $jd_{W}$ for $W=W_{\pm}(T_{p,q})$ has been calculated in \cite[Lemma 7.3]{BMT} where it was also shown that $W$ is non-adequate.
  We have $jd_{W}=4pq+2$. Now Theorem \ref{t.main} implies  that  $c(W_{\pm}(T_{p,q})>\frac{jd_{W}}{2}= 2pq+1 =2c(T_{p,q})+2 {\rm min}\{p, q\}+1>2c(T_{p,q})$.
  \end{proof}

\subsection{ Crossing number of connected sums} 
Here we give applications of Corollary \ref{criterion} to the question on additivity of crossing numbers under connected sum of knots \cite[Problems 1.67]{Kirby}.
As already mentioned, for adequate knots the crossing number is additive under connected sum. The next result proves additivity for families of knots where one summand is adequate while the other is not.

 \begin{named}{Theorem \ref{sum}} Suppose that $K$ is an adequate knot with $\mathrm{wr}(K)=0$, and let
  $K_1:=W_{\pm}(K)$. Then for any adequate knot $K_2$, the connected sum $K_1\# K_2$
is non-adequate and we have
  $$c(K_1\# K_2)=c(K_1)+c(K_2).$$
  \end{named}
  
 Before we proceed with the proof of the theorem we need some preparation. Given a knot $K$, such that for $n$ large enough the degrees of the colored Jones polynomials
  of $K$ are quadratic polynomials with rational coefficients, we will write
   $$- 4\, d_-[J_{K}(n)] =x(K)n^2+y(K)n+z(K) \ \text{and} \    4\, d_+[J_{K}(n)] =x^{*}(K)n^2+y^{*}(K)n+z^{*}(K).$$
  We also write
  $$ d_+[J_{K}(n)]-d_+[J_{K}(n)]=d_2(K )n^2+d_1(K) n+d_0(K).$$

Now let $K_1$, $K_2$ be as in the statement of Theorem \ref{sum}. 
    By   assumption and Proposition \ref{BMTprop}, for $n$ large enough the degrees of the colored Jones polynomials of both $K_1$
  and $K_2$ are quadratic polynomials.
  For the proof we need the following  elementary lemma.
  
  \begin{lem}\label{elementary} For large enough $n$, the degrees $d_{\pm}[J_{K_1\#K_2}(n)]$ are polynomials, and we have the following.
  \begin{enumerate}[(a)]
  \item $x(K_1\#K_2)=x(K_1)+x(K_2)\  {\rm{and}} \  x^{*}(K_1\#K_2)=x^{*}(K_1)+x^{*}(K_2).$
  \item  $y(K_1\#K_2)=y(K_1)+y(K_2)-2\  {\rm{and}} \  y^{*}(K_1\#K_2)=y^{*}(K_1)+y^{*}(K_2)-2.$
  \item $d_2(K_1\#K_2)=d_2(K_1)+d_2(K_2)\  {\rm{and}} \  d_1(K_1\#K_2)=d_1(K_1)+d_1(K_2)-1.$
  \end{enumerate}
  \end{lem}
  \begin{proof}
   The reduced colored Jones polynomial of a knot $K$  is defined by
  ${\overline{J_K(n)}}:= \frac{J_K(n)}{J_U(n)}$,  where $J_K(n):=J_K(n)(t)$ and $J_{U}(n):=J_U(n)(t)$ are given in Definition \ref{d.cjp}.
 Since the reduced polynomial is known to be multiplicative  under connected sum \cite{Lickorishbook},
 we get
 ${\overline{J_{K_1\#K_2}(n)}}={\overline{J_{K_1}(n)}}\cdot {\overline{J_{K_2}(n)}}$.
  The desired results follow easily since $- 4\, d_-[J_{U}(n)] =2n-2= 4\, d_+[J_{U}(n)]$.
   \end{proof}
  
  The second ingredient we need for the proof of Theorem \ref{sum} is the following lemma.
  
  \begin{lem} \label{notadequate2}Suppose that $K$ is a non-trivial adequate knot with $\mathrm{wr}(K)=0$, and let
 $K_1:=W_{\pm}(K)$. Then for any adequate knot $K_2,$ the connected sum $K_1\# K_2$
 is non-adequate.

  \end{lem} 
  \begin{proof} The claim is proved by applying the arguments applied to $K_1=W_{\pm}(K)$ in the proofs
  of Lemmas \ref{ifadequate} and \ref{notadequate} to the knot $K_1\# K_2$ and using the fact that the degrees of the colored Jones polynomial are additive under connected sum.
  For the convenience of the reader we outline the argument.
  
  First we claim that if $K_1\# K_2$ were adequate then we would have 
 \begin{align}
 \label{e.step}
 &c(K_1\# K_2)=4c(K)+1+c(K_2) ,\\
  &  g_T(K_1\# K_2)=c(K)+2 g_T(K)+g_T(K_2)-1.\notag 
    \end{align}
 
  To see this first write 
  $$ d_+[J_{K_1\#K_2}(n)]-d_-[J_{K_1\#K_2}(n)]=d_2(K_1\#K_2) n^2+d_1(K_1\#K_2) n+d_0(K_1\#K_2),$$
  then as in the  proof of Lemma \ref{notadequate}, we compute the coefficients $d_i(K_1\#K_2)$, for $i=1,2$ in two ways.

  One way to compute these coefficients  is using Lemma
  \ref{elementary} and Proposition \ref{BMTprop}. 
   By  Lemma \ref{elementary},
  $d_2(K_1\# K_2)=d_2(K_1)+d_2(K_2)$ while $d_1(K_1\# K_2)=d_1(K_1)+d_1(K_2)-1$, where
 by the calculations in  the proof of Lemma \ref{ifadequate},
  $ d_2(K_1)=2c(K)+\frac{1}{2}$ and
  $d_1(K_1)+d_2(K_1)=2-2g_T(K)-c(K)$. The corresponding quantities for $K_2$ are computed via Equation (\ref{eq:ad}).
  
   The second way to compute these coefficients, is to use Equation (\ref{eq:ad}) to obtain
  a second expression  for $d_2(K_1\# K_2) = \frac{c(K_1 \# K_2)}{2}$ and $d_2(K_1\# K_2)+d_1(K_1\# K_2) = 1-g_T(K_1 \# K_2)$. Finally, compare these two expressions to obtain the claim.

  Next apply the argument of the proof of Lemma \ref{notadequate} to show that $K_1\# K_2$ is non-adequate.
  By passing to mirror images as in the end of the proof of Lemma \ref{notadequate}, it is enough to prove that $K_1\#K_2$ with $K_1 = W_-(K)$ is non-adequate. 
To that end, start with $D=D(K)$ an adequate diagram of zero writhe and let $D_1:=W_{-}(D)$.  Also let $D_2$ be an adequate diagram of $K_2$.

As in the proof of  Lemma \ref{notadequate} conclude that $D_2\# D_1$ is a $B$-adequate diagram for $K_1\# K_2$ and that the quantities
$v_B(D_1\#D_2)=(2v_B(D)+1)+v_B(D_2)-1 = 2v_B(D)+v_B(D_2)$ and $c_+(D_1\#D_2) =4c_+(D)+c_+(D_2)$ are invariants of $K_1\# K_2$.

Next suppose,  for a contradiction, that $K_1\# K_2$ is adequate and let ${\bar D}$ be an adequate diagram.  We have 
$$
v_B({\bar D})=v_B(D_1\#D_2)=2v_B(D)+v_B(D_2) \ { \rm{and}}\ 
c_+({\bar D})=4c_+(D)+c_+(D_2).
$$
By \cite{Abe} and Equations \eqref{e.step} we have
  \begin{equation}
  g_T({\bar D})=g_T(K_1\# K_2)=c(K)+2 g_T(D)+g_T(D_2)-1.
 \label{eq:35}
  \end{equation}
  By   \eqref{e.step}, and the fact that adequate diagrams realize the  knot crossing number,  $c({\bar D})=4c(D)+1+c(D_2)$ and $c(K)=c(D)$.
 
  Now using the definition of the Turaev genus of knot diagrams to expand the leftmost  and the rightmost sides of Equation (\ref{eq:35}) we get $v_A({\bar D})=2v_A(D)+v_A(D_2)-1$.
  
  Next we will calculate the quantity  $y(K_1\# K_2)$ of Lemma \ref{elementary} in  two ways:
   
   Firstly, since we assumed that ${\bar D}$ is an adequate diagram for $K_1\# K_2$,  applying  Equation \eqref{top},  we get $y(K_1\# K_2)=2(c({\bar D}) -v_A({\bar D}))=8 c(D)+2c(D_2)-4 v_A(D) -2v_A(D_2)+4$.
   
   Secondly, by Lemma \ref{elementary}, we get  $y(K_1\# K_2)=y(K_1)+y( K_2)-1$, which combined with 
   Equations  (\ref{eq:34}) and (\ref{top}) gives  $y(K_1\# K_2)=8 c(D)-4 v_B(D) +2c(D_2)-2v_A(D_2)-1$. 
  We note that in order for the two resulting expressions for $y(K_1\# K_2)$ to be equal we must have
  $4v_A(D)+4=-4 v_B(D)-1$ or $-1\equiv 0\mod 4$, which is absurd.
  We conclude that $K_1\#K_2$ is non-adequate.
    \end{proof}
  
  Now we give the proof of Theorem \ref{sum}.
  
  \begin{proof} Note that if $K$ is the unknot then so is $W_{\pm}(K)$ and the result follows trivially.
  Suppose that $K$ is a non-trivial knot. Then by Lemma \ref{notadequate2} we obtain that $K_1\#K_2$ is non-adequate.
  
  As discussed above $jd_{K_1}=2 (4c(K)+1)=2(c(W_{\pm}(D))-1)$. On the other hand, $jd_{K_2}=2c(D_2)=2c(K)$ where $D_2$ is an adequate diagram for $K_2$.
  Hence, by Lemma \ref{elementary},  $jd_{K_1\#K_2}= jd_{K_1}+ jd_{K_2}=2(c(W_{\pm}(D))+c(D_2)-1),$ and setting $D_1=W_{\pm}(D)$ we obtain
$jd_{K_1\#K_2}=2(c(D_1\#D_2)-1)$. 
 Thus by Corollary \ref{criterion},  we obtain that
$c(K_1\#K_2)=c(D_1\#D_2)=c(D_1)+c(D_2)=c(K_1)+c(K_2)$, where the last equality follows since, by Theorem \ref{doubleadequate}, we have  $c(K_1)=c(D_1)=c(W_{\pm}(D)).$
  \end{proof}
  
  \begin{remark} In \cite{BMT3} Baker, Motegi and Takata computed the Jones slopes of  \emph{Mazur doubles} of adequate knots. Then they use the methods of this section to show that if $K$ is an adequate knot with crossing number $c(K)$ and writhe $\mathrm{wr}(K)$, then the crossing number of the Mazur double of $K$ is either $9c(K)+2$ or $9c(K)+3$.
   \end{remark}

\bibliographystyle{plain}
\bibliography{references}

\begin{thebibliography}{10}

\bibitem{Abe}
T.~Abe.
\newblock The {T}uraev genus of an adequate knot.
\newblock {\em Topology Appl.}, 156(17):2704--2712, 2009.

\bibitem{Adams}
C.~C. Adams.
\newblock {\em The knot book}.
\newblock W. H. Freeman and Company, New York, 1994.
\newblock An elementary introduction to the mathematical theory of knots.

\bibitem{BMT}
K.~L. Baker, K.~Motegi, and T.~Takata.
\newblock The strong slope conjecture for twisted generalized {W}hitehead
  doubles.
\newblock {\em Quantum Topol.}, 11(3):545--608, 2020.

\bibitem{BMT3}
Kenneth~L. Baker, Kimihiko Motegi, and Toshie Takata.
\newblock The strong slope conjecture and crossing numbers for mazur doubles of
  knots.
\newblock GT, arxiv:2204.05725, 2022.

\bibitem{blog}
R.~Budney.
\newblock Low dimensional topology.
\newblock URL:
  \url{https://ldtopology.wordpress.com/2012/03/25/crossing-numbers-of-knots/},
  Current Month 2012.

\bibitem{5author}
O.~T. Dasbach, D.~Futer, E.~Kalfagianni, Xiao-Song Lin, and Neal~W. Stoltzfus.
\newblock The {J}ones polynomial and graphs on surfaces.
\newblock {\em J. Combin. Theory Ser. B}, 98(2):384--399, 2008.

\bibitem{Garslopes}
S.~Garoufalidis.
\newblock The {J}ones slopes of a knot.
\newblock {\em Quantum Topol.}, 2(1):43--69, 2011.

\bibitem{Hoste}
J.~Hoste, M.~B. Thistlethwaite, and J.~Weeks.
\newblock The first 1,701,936 knots.
\newblock {\em Math. Intelligencer}, 20(4):33--48, 1998.

\bibitem{Indiana}
E.~Kalfagianni.
\newblock A {J}ones slopes characterization of adequate knots.
\newblock {\em Indiana Univ. Math. J.}, 67(1):205--219, 2018.

\bibitem{KTran}
E.~Kalfagianni and A.~T. Tran.
\newblock Knot cabling and the degree of the colored {J}ones polynomial.
\newblock {\em New York J. Math.}, 21:905--941, 2015.

\bibitem{Kauffmanstates}
L.~H. Kauffman.
\newblock State models and the {J}ones polynomial.
\newblock {\em Topology}, 26(3):395--407, 1987.

\bibitem{KauffmanLins}
L.~H. Kauffman and S.~L. Lins.
\newblock {\em Temperley-{L}ieb recoupling theory and invariants of
  {$3$}-manifolds}, volume 134 of {\em Annals of Mathematics Studies}.
\newblock Princeton University Press, Princeton, NJ, 1994.

\bibitem{Kirby}
W.~H. Kazez, editor.
\newblock {\em Geometric topology}, volume~2 of {\em AMS/IP Studies in Advanced
  Mathematics}. American Mathematical Society, Providence, RI; International
  Press, Cambridge, MA, 1997.

\bibitem{Lackenby}
Marc Lackenby.
\newblock Elementary knot theory.
\newblock In {\em Lectures on geometry}, Clay Lect. Notes, pages 29--64. Oxford
  Univ. Press, Oxford, 2017.

\bibitem{lee2020jones}
C.~R.~S. Lee.
\newblock Jones slopes and coarse volume of near-alternating links.
\newblock {\em Comm. Anal. Geom.}, 30(4):891--948, 2022.

\bibitem{Lickorishbook}
W.~B.~R. Lickorish.
\newblock {\em An introduction to knot theory}, volume 175 of {\em Graduate
  Texts in Mathematics}.
\newblock Springer-Verlag, New York, 1997.

\bibitem{LickorishThistlethwaite}
W.~B.~R. Lickorish and M.~B. Thistlethwaite.
\newblock Some links with nontrivial polynomials and their crossing-numbers.
\newblock {\em Comment. Math. Helv.}, 63(4):527--539, 1988.

\bibitem{Knotinfo}
C.~Livingston and A.~H. Moore.
\newblock {KnotInfo: Table of Knot Invariants}.
\newblock URL: \url{knotinfo.math.indiana.edu}, June 2021.

\bibitem{MasbaumVogel}
G.~Masbaum and P.~Vogel.
\newblock {$3$}-valent graphs and the {K}auffman bracket.
\newblock {\em Pacific J. Math.}, 164(2):361--381, 1994.

\bibitem{Murasugi}
K.~Murasugi.
\newblock Jones polynomials and classical conjectures in knot theory.
\newblock {\em Topology}, 26(2):187--194, 1987.

\bibitem{Pascual}
A.~J. Pascual.
\newblock On wrapping number, adequacy and the crossing number of satellite
  knots.
\newblock arXiv:1712.05635, 2017.

\bibitem{Przytycki}
J.~H. Przytycki.
\newblock Skein modules of {$3$}-manifolds.
\newblock {\em Bull. Polish Acad. Sci. Math.}, 39(1-2):91--100, 1991.

\bibitem{Russell}
H.~M. Russell.
\newblock The {B}ar-{N}atan skein module of the solid torus and the homology of
  {$(n,n)$} {S}pringer varieties.
\newblock {\em Geom. Dedicata}, 142:71--89, 2009.

\bibitem{spanning}
M.~B. Thistlethwaite.
\newblock A spanning tree expansion of the {J}ones polynomial.
\newblock {\em Topology}, 26(3):297--309, 1987.

\bibitem{alternating}
M.~B. Thistlethwaite.
\newblock Kauffman's polynomial and alternating links.
\newblock {\em Topology}, 27(3):311--318, 1988.

\bibitem{Turaevgenus}
V.~G. Turaev.
\newblock A simple proof of the {M}urasugi and {K}auffman theorems on
  alternating links.
\newblock {\em Enseign. Math. (2)}, 33(3-4):203--225, 1987.

\bibitem{Wenzl}
H.~Wenzl.
\newblock On sequences of projections.
\newblock {\em C. R. Math. Rep. Acad. Sci. Canada}, 9(1):5--9, 1987.

\end{thebibliography}

\end{document}